\documentclass[11pt]{article}

\usepackage{geometry}
\usepackage{amssymb,amsmath,amsthm}
\usepackage[T1]{fontenc}
\usepackage[utf8]{inputenc}
\usepackage{bbm}
\usepackage{color}
\usepackage{cancel}
\usepackage{ulem}
\usepackage{todonotes}

\def\Ex{{\mathbb E}}
\def\Prob{{\mathbb P}}
\def\er{{\mathbb R}}
\def\supp{\mathrm{supp}}
\def\ve{\varepsilon}
\def\ind{\mathbbm{1}}
\def\bfi{\mathbf{i}}
\def\bfl{\mathbf{l}}
\def\zet{{\mathbb Z}}

\newtheorem{theorem}{Theorem}[section]
\newtheorem{lemma}[theorem]{Lemma}
\newtheorem{conjecture}[theorem]{Conjecture}
\newtheorem{proposition}[theorem]{Proposition}
\newtheorem{corollary}[theorem]{Corollary}

\theoremstyle{remark}
\newtheorem{rem}[theorem]{Remark}

\title{Norms of Randomized Circulant Matrices 
\thanks{Supported by the National Science Centre, Poland grant 2015/18/A/ST1/00553}}
\author{Rafał Latała and Witold Świątkowski}
\date{}

\begin{document}
\maketitle

\begin{abstract}
We investigate two-sided bounds for operator norms of  random matrices with  non-homogenous independent entries. We formulate a lower bound for Rademacher matrices
and conjecture that it may be reversed up to a universal constant. 
We show that our conjecture holds up to $\log\log n$ factor for randomized $n\times n$ circulant matrices
and that the double logarithm may be eliminated under some mild additional assumptions on the coefficients.
\end{abstract}

\section{Introduction and main results}

Study of random matrices is one of the central issues of probability theory and its applications.
Classical random matrix theory, initially motivated by mathematical physics,  is mostly concerned with the investigation of homogenous matrix ensembles, possesing a large degree of symmetry \cite{AGZ,Tao}. In many applications one needs however to consider highly non-homogenous random matrices. In such situations one cannot expect as precise results  as for the classical ensembles, nevertheless in recent years there was made a significant progress in this area and many important estimates were derived, cf. \cite{BBvH,Tr,vHstr} and references therein.

The aim of this paper is to discuss bounds for the operator norm of non-homogenous random matrices $X=(X_{ij})_{i,j\leq n}$ with independent entries.  It is easy to reduce to the case of mean zero random matrices, i.e. when $X_{ij}$ are independent
centered r.v's. The Gaussian case was solved in \cite{HLY}, where it was shown that if 
$X_{ij}\sim \mathcal{N}(0,\sigma_{ij}^2)$ are independent Gaussian r.v's,
\begin{align*}
\Ex\|(X_{ij})\|
&\sim \Ex\max_{1\leq i\leq n}\sqrt{\sum_{j=1}^n X_{ij}^2}+\Ex\max_{1\leq j\leq n}\sqrt{\sum_{i=1}^n X_{ij}^2}
\\
&\sim \max_i\sqrt{\sum_{j=1}^n \sigma_{ij}^2}+\max_j\sqrt{\sum_{i=1}^n \sigma_{ij}^2}
+\max_{1\leq k\leq n}\min_{I\subset [n],|I|\leq k}\max_{i,j\notin I}\sqrt{\log (k+1)}|\sigma_{ij}|.
\end{align*}
Here and throughout the paper, $\|\cdot\|$ denotes the operator norm, unless indicated otherwise.
The last estimate above  differs slightly from the one formulated in \cite{HLY}, but it is not hard to see that it is equivalent to it (see the proof of the second bound in Proposition \ref{prop:uppermaxrows} below).

The most interesting case left are Rademacher matrices, i.e. random matrices
with coefficients $X_{ij}=a_{ij}\ve_{ij}$, where $\ve_{ij}$, $1\leq i,j\leq n$ are independent symmetric $\pm 1$ r.v's. 
The main body of the paper consists of results proved in this setting. A lot of things may be however done
in a bigger generality,  as we show in Section \ref{sec:ext}.

Our first result is the lower bound for the operator norm. For two nonnegative functions $f$ and $g$ we write $f\gtrsim g$ (or $g\lesssim f$),
if there exists an absolute constant $C$ such that $ Cf \geq g$. Notation $f\sim g$ means that $f\gtrsim g$ and $g\gtrsim f$. 
 We use $C$ and $c$ to denote universal constants and their values might be different at each appearance.
We write
$\|S\|_p=(\Ex|S|^p)^{1/p}$ for the $L_p$-norm of a random variable $S$. The same notation is used for the $\ell_p$-norm of a vector: 
$\|x\|_p=(|x_1|^p+\ldots+|x_d|^p)^{1/p}$, where $x=(x_1,\ldots,x_d)\in\mathbb{R}^d$.

\begin{theorem}
\label{thm:lower}
Let $(a_{ij})_{i,j\le n}$ be any  real matrix and $X_{ij}=a_{ij}\varepsilon_{ij}$. Then
\begin{align}
\notag
\Ex\|(X_{ij})_{i,j\leq n}\|
&\gtrsim
\max_{1\leq i\leq n}\left(\sum_{j=1}^n \Ex X_{ij}^2\right)^{1/2}+\max_{1\leq j\leq n}\left(\sum_{i=1}^n \Ex X_{ij}^2\right)^{1/2}
\\
\label{eq:lower}
&\phantom{=}
+\max_{1\leq k\leq n} \min_{I\subset [n],|I|\leq k}
\sup_{\|s\|_2,\|t\|_2\leq 1}\left\|\sum_{i,j\notin I}X_{ij}s_it_j\right\|_{\log (k+1)}.
\end{align}
\end{theorem}

\medskip

\noindent
{\bf Remark.} Since $\|\mathcal{N}(0,\sigma^2)\|_p\sim \sqrt{p}\sigma$ for $p\geq \ln 2$, in the Gaussian case  we have
\[
\sup_{\|s\|_2,\|t\|_2\leq 1}\left\|\sum_{i,j}\sigma_{ij}g_{ij}s_it_j\right\|_p
\sim \sqrt{p}\max_{i,j}|\sigma_{ij}|,
\]
where $g_{ij}$ are i.i.d. $\mathcal{N}(0,1)$ r.v's.
Thus the main result of \cite{HLY} states that \eqref{eq:lower} may be reversed if $X_{ij}$ are independent
centered Gaussian r.v's.

Theorem \ref{thm:lower} and the remark above motivate the following conjecture.

\begin{conjecture}
\label{conj:rademacher}
We have
\begin{align*}
\Ex\|(a_{ij}\ve_{ij})_{i,j\leq n}\|
&\sim
\max_{1\leq i\leq n}\left(\sum_{j=1}^n a_{ij}^2\right)^{1/2}+\max_{1\leq j\leq n}\left(\sum_{i=1}^n a_{ij}^2\right)^{1/2}
\\
&\phantom{=}
+\max_{1\leq k\leq n}\min_{I\subset [n],|I|\leq k}
\sup_{\|s\|_2,\|t\|_2\leq 1}\left\|\sum_{i,j\notin I}a_{ij}\ve_{ij}s_it_j\right\|_{\log (k+1)}.
\end{align*}
\end{conjecture}

Seginer \cite{Se} proved that  
\begin{equation}
\label{eq:Seg}
\Ex\|(a_{ij}\ve_{ij})_{i,j\leq n}\|\lesssim  
\sqrt[4]{\log (n+1)}\Bigg(\max_{1\leq i\leq n}\sqrt{\sum_{j=1}^n a_{ij}^2}+\max_{1\leq j\leq n}\sqrt{\sum_{i=1}^n a_{ij}^2}\Bigg)
\end{equation}
and that one cannot improve the  factor $\log^{1/4}(n+1)$. This shows that
Conjecture \ref{conj:rademacher} holds up to $\log^{1/4}(n+1)$ factor.

Our main result states that if coefficients $(a_{ij})$ form a circulant matrix then Conjecture \ref{conj:rademacher} is satisfied up to $\log \log n$ factor.

\begin{theorem}
\label{thm:twosidedbound}
Suppose that $(a_{ij})_{i,j\leq n}$ is a circulant matrix, i.e. $a_{ij}=b_{i-j \mod n}$ for a deterministic sequence 
$(b_i)_{i=0}^{n-1}$. Then
\begin{align*}
\sqrt{\sum_{j=0}^{n-1}b_j^2}&+\sup_{\|s\|_2,\|t\|_2\leq 1}\left\|\sum_{i,j=1}^na_{ij}\ve_{ij}s_it_j\right\|_{\log (n+1)}
\lesssim \Ex\|(a_{ij}\ve_{ij})\|
\\
&\lesssim \sqrt{\log\log (n+3)}
\sqrt{\sum_{j=0}^{n-1}b_j^2}+\log\log(n+3)\sup_{\|s\|_2,\|t\|_2\leq 1}\left\|\sum_{i,j=1}^na_{ij}\ve_{ij}s_it_j\right\|_{\log (n+1)}.
\end{align*}
Moreover, $\log\log$ factors may be eliminated in the case when $b_i$ takes only values $0$ and $1$ (i.e. when $(a_{ij})$ is an adjacency matrix of a directed circulant graph).
\end{theorem}

In order to apply such a result it would be nice to have a simple two-sided bound for the quantity 
\begin{equation}
\label{eq:defAvep}
\|A\|_{\ve,p}:=\sup_{\|s\|_2,\|t\|_2\leq 1}\left\|\sum_{i,j=1}^na_{ij}\ve_{ij}s_it_j\right\|_{p}.
\end{equation}
Two-sided bounds for $L_p$-norms of Rademacher sums were derived in \cite{Hi} on the base of tail bounds \cite{MS} (see also \cite{HK} for a discussion of various equivalent norms):
\[
\Bigg\|\sum_{k=1}^n a_k\ve_k\Bigg\|_p\sim \sum_{k\leq p}a_k^*+\sqrt{p}\Bigg(\sum_{k>p}(a_k^*)^2\Bigg)^{1/2}\sim
\sup\Bigg\{\sum_{k=1}^n a_kb_k\colon \|b\|_\infty\leq 1, \|b\|_2\leq \sqrt{p}\Bigg\},
\]
where $(a_k^*)_{k=1}^n$ denotes the nonincreasing rearrangement of $(|a_k|)_{k=1}^n$. It is however nonobvious how to apply 
the above bounds to get a simple two-sided  estimate for $\|A\|_{\ve,p}$. We were able to derive such bound when $A$ is an adjacency matrix of a (directed) graph.

\begin{proposition}
\label{prop:szacrad}
For any $E\subset [n]\times [n]$ and $p\geq 1$ we have
\begin{equation}
\label{eq:estadj}
\|\ind_E\|_{\ve,p}=\sup_{\|s\|_2,\|t\|_2\leq 1}\Bigg\|\sum_{(i,j)\in E}\ve_{ij}s_it_j\Bigg\|_p 
\sim \max_{I\subset E, |I|\leq p}\|(\ind_{\{(i,j)\in I\}})\|.
\end{equation}
\end{proposition}

In the general case, a similar result holds. However, here the upper and lower bounds are not of the same order.
\begin{proposition}
\label{prop:szacradA}
For any matrix $A=(a_{ij})_{i,j\le n}$ and $p\geq 1$,
\[
\frac{1}{2}\max_{I\subset [n]\times[n], |I|\leq p}\|(|a_{ij}|\ind_{(i,j)\in I})\|
\leq 
\|A\|_{\ve,p}\lesssim \sqrt{\ln(p+1)} \max_{I\subset [n]\times[n], |I|\leq p}\|(|a_{ij}|\ind_{(i,j)\in I})\|.
\]
\end{proposition}
We do not know whether the logarithmic factor is necessary. 

\medskip

\textbf{Organization of the paper.} In Section \ref{sec:proofs} we prove the main results of the paper, i.e. Theorems \ref{thm:lower} and \ref{thm:twosidedbound}.
In Section \ref{sec:radnorms} we study norms $\|A\|_{\ve,p}$ and establish  Propositions \ref{prop:szacrad} and \ref{prop:szacradA}. We also provide there estimates for  $\|A\|_{\ve,p}$ in the case when $A$ are adjacency matrices of hypercubes $\zet_2^d$  and more general discrete tori $\zet_m^d$.
The adjacency matrix of $\zet_m^d$ is not circulant, however it is very close to have such a property and 
Corollary \ref{cor:tor} shows that we may apply previously derived estimates to get two-sided bounds
on $\Ex\|(\ind_{i\sim j}\ve_{ij})_{i,j\in \zet_m^d}\|$. In the last section of the paper we extend 
Theorem \ref{thm:lower} and Conjecture \ref{conj:rademacher} to the case of random matrices with independent entries $X_{ij}$ satisfying the mild regularity condition $\|X_{ij}\|_{2p}\leq \alpha\|X_{ij}\|_p$ for $p\geq 2$. We show that results of \cite{HLY} imply that more general Conjecture \ref{conj:main} holds in the case when $X_{ij}$ are mixtures of Gaussian variables. We also show that to establish formulated conjectures  it is enough to prove a slightly weaker bound \eqref{eq:reduction}.

\section{Proofs of main results}
\label{sec:proofs}

In the sequel we will frequently compare $L_p$-norms of real and vector-valued Rademacher sums
$S=\sum_{k=1}^nx_k\ve_k$, where
$x_i$ are vectors from a normed space $(F,\|\ \|)$ and $\ve_i$ are independent symmetric $\pm 1$ r.v's.
The classical Khintchine (in the case when coefficients $x_i$ are real) and the Kahane-Khintchine inequalities, cf. \cite[Section 4.3]{LT}, state that for $p>q>0$,
\[
\|S\|_q\leq \|S\|_p\leq C_{p,q}\|S\|_q,
\]
where $\|S\|_q=(\Ex\|S\|^q)^{1/q}$ and  $C_{p,q}$ is a constant depending only on $p$ and $q$.
Moreover for $p>q>1$, $C_{p,q}\leq \sqrt\frac{p-1}{q-1}$, therefore for $p\geq 2$, $C_{p,2}\leq \sqrt{p}$ and
for $p\geq 1+(e-2)^{-1}$, $C_{2p,p}\leq \sqrt{e}$.

Markov's inequality yields $\Prob(\|S\|\geq e\|S\|_p)\leq e^{-p}$. Using the Paley-Zygmund inequality
\[
\Prob\Big(Z\geq\frac{1}{2} \Ex Z\Big)\geq \frac{1}{4} \frac{(\Ex Z)^2}{\Ex Z^2} \quad\mbox{for nonnegative r.v. $Z$}, 
\]
one may derive a reverse lower tail bound for $\|S\|$ (similar estimates can be found in \cite{DMS, LaSt}, we present details for the sake of completeness). For any $p\ge 1+(e-2)^{-1}$, by taking $Z=\|S\|^p$ we obtain
\begin{equation}
\label{eq:LowtailRad}
\Prob\Big(\|S\|\geq \frac{1}{2}\|S\|_p\Big)\geq \frac{1}{4}\left(\frac{\|S\|_p}{\|S\|_{2p}}\right)^{2p}\geq \frac{1}{4}e^{-p}.
\end{equation}

\subsection{Proof of Theorem \ref{thm:lower}}

We start with a simple observation.

\begin{lemma}
\label{lem:smallsupportrad}
For any $p\geq 2$,
\begin{align*}
&\sup_{\|s\|_2,\|t\|_2\leq 1}\left\|\sum_{ij}a_{ij}\ve_{ij}s_it_j\right\|_p
\\
&\leq \sup_{\substack{\|s\|_2\leq 1,|\supp(s)|\leq p,\\ \|t\|_2\leq 1, |\supp(t)|\leq p}}
\left\|\sum_{ij}a_{ij}\ve_{ij}s_it_j\right\|_p
+\max_{i}\left(\sum_{j}a_{ij}^2\right)^{1/2}+\max_{j}\left(\sum_{i}a_{ij}^2\right)^{1/2}.
\end{align*}
\end{lemma}

\begin{proof}
For a vector $z\in \er^n$ put $I_p(z):=\{i\leq n\colon\ |z_i|^2\geq p^{-1}\}$. If $\|z\|_2\le 1$, then $|I_p(z)|\leq p$, so
\[
\sup_{\|s\|_2,\|t\|_2\leq 1}\left\|\sum_{i\in I_p(s),j\in I_p(t)}a_{ij}\ve_{ij}s_it_j\right\|_p
\leq \sup_{\substack{\|s\|_2\leq 1,|\supp(s)|\leq p,\\ \|t\|_2\leq 1, |\supp(t)|\leq p}}
\left\|\sum_{ij}a_{ij}\ve_{ij}s_it_j\right\|_p.
\]
The Khintchine inequality yields
\begin{align*}
\sup_{\|s\|_2,\|t\|_2\leq 1}\left\|\sum_{i\leq n,j\notin I_p(t)}a_{ij}\ve_{ij}s_it_j\right\|_p
&\leq \sup_{\|s\|_2,\|t\|_2\leq 1}\Big(p\sum_{i\leq n,j\notin I_p(t)}a_{ij}^2s_i^2t_j^2\Big)^{1/2}
\\
&\leq \sup_{\|s\|_2\leq 1}\Big(\sum_{i,j\leq n}a_{ij}^2s_i^2\Big)^{1/2}
=\max_{i}\left(\sum_{j}a_{ij}^2\right)^{1/2}.
\end{align*}
In a similar way we show that
\[
\sup_{\|s\|_2,\|t\|_2\leq 1}\left\|\sum_{i\notin I_p(s),j\in I_p(t)}a_{ij}\ve_{ij}s_it_j\right\|_p
\leq \max_{j}\left(\sum_{i}a_{ij}^2\right)^{1/2}.
\]
\end{proof}

\begin{proof}[Proof of Theorem \ref{thm:lower}]

The estimate
\begin{equation}
\label{eq:lower1}
\left(\Ex\|(X_{ij})\|^2\right)^{1/2}\geq
\max\Bigg\{ \max_{i}\left(\sum_{j} \Ex X_{ij}^2\right)^{1/2}, \max_{j}\left(\sum_{i} \Ex X_{ij}^2\right)^{1/2}\Bigg\}
\end{equation}
is trivial. Moreover, by the Khintchine-Kahane inequality,
$\Ex\|(a_{ij}\ve_{ij})\|\sim (\Ex\|(a_{ij}\ve_{ij})\|^2)^{1/2}$.

To establish the last term in the lower bound let us fix $1\leq k\leq n$. We need to show that
\[
\gamma:=\min_{I\subset [n],|I|\leq k}
\sup_{\|s\|_2,\|t\|_2\leq 1}\left\|\sum_{i,j\notin I}X_{ij}s_it_j\right\|_{\log (k+1)}\leq C
\Ex\|(X_{ij})\|.
\]

Observe that by the Khintchine inequality
\begin{align*}
\gamma
&\leq 
C\sqrt{\log(k+1)}\sup_{\|s\|_2,\|t\|_2\leq 1}\left(\sum_{ij}s_i^2t_j^2a_{ij}^2\right)^{1/2}
= C\sqrt{\log(k+1)}\max_{ij}|a_{ij}|
\\
&\leq C\sqrt{\log(k+1)}\Ex\|(X_{ij})\|,
\end{align*}
so we may  consider only $k$ large, in particular we may assume that $2\log(k+1)\leq \sqrt{k}$ and $\frac{1}{2}\log(k+1)>1+(e-2)^{-1}$.

If 
\[
\gamma\leq 2\Bigg(\max_{i}\Bigg(\sum_{j} a_{ij}^2\Bigg)^{1/2}
+\max_{j}\Bigg(\sum_{i} a_{ij}^2\Bigg)^{1/2}\Bigg)
\]
the estimate follows from the trivial bound \eqref{eq:lower1}. In the opposite case
for any $|I|\leq k$ by Lemma \ref{lem:smallsupportrad} 
(applied to $\sum_{i,j\notin I}$ instead of $\sum_{i,j}$) there exists a set $J$ of cardinality at most 
$2\log(k+1)\leq \sqrt{k}$
disjoint from $I$ and unit vectors $t,s$ such that
\[
\left\|\sum_{i,j\in J}X_{ij}t_is_j\right\|_{\log (k+1)}\geq \frac{1}{2}\gamma.
\]

Thus we may inductively construct disjoint sets  $I_l$ and unit vectors 
$s^{(l)}, t^{(l)}$, $1\leq l\leq \sqrt{k}$ such that 
\[
\|S_l\|_{\log (k+1)}\geq \frac{1}{2}\gamma,\quad \mbox{ where }\quad
S_l:=\sum_{i,j\in I_l} s_i^{(l)}t_j^{(l)}X_{ij}. 
\]

Let $p=\frac{1}{2}\log(k+1)$. Then $p\geq 1+(e-2)^{-1}$ and by the Khintchine inequality $\|S\|_{2p}\le\sqrt{e}\|S\|_p$. Thus the lower tail bound for Rademacher sums \eqref{eq:LowtailRad} yields
\begin{equation}
\label{eq:SlTail}
\Prob\Big(|S_l|\geq \frac{1}{4\sqrt{e}}\gamma\Big)\geq 
\Prob\Big(|S_l|\geq\frac{1}{2\sqrt{e}} \|S_l\|_{\log(k+1)}\Big)\geq
\Prob\Big(|S_l|\geq\frac{1}{2}\|S_l\|_{\frac{1}{2}\log(k+1)}\Big)\geq\frac{c}{\sqrt{k}}.
\end{equation}
Hence we have
\begin{align*}
\Ex\|(X_{ij})\|&\geq\frac{1}{4}\gamma\Prob\Big(\max_{l\leq\sqrt{k}}|S_l|\geq\frac{1}{4}\gamma\Big)=\frac{1}{4}\gamma\left(1-\prod_{l\leq\sqrt{k}}\Prob\Big(|S_l|<\frac{1}{4}\gamma\Big)\right)\\
&\geq\frac{1}{4}\gamma\left(1-\Big(1-\frac{c}{\sqrt{k}}\Big)^{\sqrt{k}}\right)\geq 
\frac{1}{4}(1-e^{-c})\gamma,
\end{align*}
where the consecutive steps follow by Chebyshev's inequality, independence of the r.v's $S_l$, the tail bound \eqref{eq:SlTail} and the inequality $1-y\leq e^{-y}$.
\end{proof}

\subsection{Proof of Theorem \ref{thm:twosidedbound}}

Let $1\le p_1\le\ldots\le p_d\le n/2$ be fixed natural numbers. Consider a circulant graph $G=(V,E)$ where $V=[n]$ and $(i,j)\in E$ if and only if $i-j=\pm p_k \mod n$ for some $1\le k\le d$. Although $G$ is an undirected graph, we treat it as a directed graph for notational simplicity. It is a regular graph of degree $2d$ or $2d-1$, when $p_d\neq n/2$ or $p_d=n/2$ respectively. Therefore it has either $|E|=2dn$ or $(2d-1)n$ (directed) edges. The matrix $(\ind_E(i,j))_{i,j\le n}$ is the adjacency matrix of $G$. For simplicity of notation we will denote it by $\ind_E$. If $I\subset V$, then we will also write just $I$ for a subgraph $(I,(I\times I)\cap E)\subset G$.

For fixed $k\in[n]$ we introduce the following two subsets of $[n]$:
\begin{align*}
U_k&=\{l\in[n]: l=k+\sum_{i\in I}p_i \mod n\ \mathrm{for\ some}\ \emptyset\subseteq I\subseteq [d]\},\\
D_k&=\{l\in[n]: l=k-\sum_{i\in I}p_i \mod n\ \mathrm{for\ some}\ \emptyset\subseteq I\subseteq [d]\}.
\end{align*}
Note that the cardinalities of $U_k$ and $D_k$ do not depend on $k$ and they are both equal $m\le 2^d$. Observe also that for any 
$l\in D_k$, there are at least $d$ distinct elements of $D_k$ connected to $l$ with an edge. Moreover $i\in D_k$ if and only if $k\in U_i$, hence any $i\in [n]$ belongs to exactly $m$ of the sets $D_k,k\in[n]$.

There is a significant similarity between $U_k,D_k$ (as subgraphs of $G$) and the hypercube $\zet_2^d$.  
If $\sum_{i\in I}p_i\neq\sum_{j\in J}p_j \mod n$ whenever $I,J\subseteq[n]$ and $I\neq J$, then maps $\ind_I\mapsto k\pm\sum_{i\in I} p_i \ \mod n$ are
 isomorphisms between $\zet_2^d$ and $U_k$ or $D_k$ respectively.  
 Otherwise, the maps are not injective and two vertices of $\zet_2^d$ may be pasted into one vertex of $U_k$ or $D_k$, which inherits neighbours of both. Nevertheless, the degree of a vertex in $U_k$ or $D_k$ never exceeds its degree in $G$, which is at most $2d$. Due to this structural similarity, we will refer to $U_k$ and $D_k$ as `the upper cube' and `the lower cube' at $k$ respectively.

For a fixed sequence $k_1,\ldots,k_s\in[n]$ we define the modified, disjoint version of the lower cubes:
\[
I_1=D_{k_1},\qquad I_{l}=D_{k_l}\setminus(\bigcup_{r<l}D_{k_r}),\quad l=2,\ldots,s.
\]
We are going to show that $k_1,\ldots,k_s$ can be chosen in such way that $\frac{|I_l|}{|D_{k_l}|}=\frac{|I_l|}{m}\ge\frac{7}{8}$ for any $l$, while at least $\frac{1}{32}$ of the edges from $E$ connects vertices belonging to some $I_l, 1\le l\le s$.

\begin{lemma}
\label{lem:exclusion}
Fix a number $c\in(0,1)$ and a set $J\subset[n]$ of cardinality $|J|\le cn$. Then there is $k\in[n]$ such that $|D_k\setminus J|\ge(1-c)m$.
\end{lemma}
\begin{proof}
Recall that any $i\in[n]$ belongs to exactly $m$ of the sets $D_k,k\in[n]$. Therefore
\[
(1-c)n\le|[n]\setminus J|=\sum_{i\in[n]\setminus J}\sum_{k=1}^n \frac{1}{m}\ind_{D_k}(i)=\frac{1}{m}\sum_{k=1}^n \sum_{i\in[n]\setminus J}\ind_{D_k}(i)=\frac{1}{m}\sum_{k=1}^n |D_k\setminus J|.
\]
\end{proof}

\begin{lemma}
\label{lem:goodsequence}
There are $s\geq\frac{n}{8m}$ and $k_1,\ldots,k_s\in[n]$ such that $|I_l|\ge\frac{7}{8}m$ for $l=1,\ldots,s$. Moreover,
\[
\left|\bigcup_{l=1}^s \left(I_l\times I_l\right)\cap{E}\right|\ge\frac{dn}{16}.
\]
\end{lemma}

\begin{proof} 
Let $s:=\lceil \frac{n}{8m}\rceil$. We construct inductively $k_1,\ldots,k_s$. For $k_1$ we choose arbitrary element of $[n]$.
Assume that $k_1,\ldots,k_r$ are chosen and $r<\frac{n}{8m}$. Then the set $J=\bigcup_{l\le r}D_{k_l}$ has at most 
$rm<\frac{1}{8}n$ elements. Therefore, by Lemma \ref{lem:exclusion}, there is $k_{r+1}\in[n]$ such that $D_{k_{r+1}}\setminus J$ has at least $\frac{7}{8}m$ elements.

Notice that since $I_l\subseteq D_{k_l}$, we have $|D_{k_l}\setminus I_l|=|D_{k_l}|-|I_l|\le\frac{1}{8}m$. Recalling that $|D_{k_l}|=m$ and each vertex has degree at most $2d$, we have
\begin{align*}
|(I_l\times I_l)\cap{E}|&\ge|(D_{k_l}\times D_{k_l})\cap{E}|-|([n]\times(D_{k_l}\setminus I_l))\cap{E}|-|((D_{k_l}\setminus I_l)\times[n])\cap{E}|\\
&\ge dm-2d\frac{m}{8}-2d\frac{m}{8}=\frac{dm}{2}.
\end{align*}
Since $I_l,1\le l\le s$ are disjoint and $s>\frac{n}{8m}$, we obtain the result.
\end{proof}

It follows that the graph $G$ contains a subgraph $G'=(V,E')$ which consists of $s$ mutually unconnected parts 
{$(I_l, (I_l\times I_l)\cap E),l=1,\ldots,s$ having at most $2^d$ vertices each, and it contains at least 3\% of the edges from $E$. In other words, the incidence matrix $I'$ of the graph $G'$ is a block diagonal matrix (possibly after permutation of rows and columns), which cuts out at least 3\% of ones from $\ind_{{E}}$. We are going to further improve this result. In what follows we write $(a_{ij})\le (b_{ij})$ if $a_{ij}\le b_{ij}$ for any $i,j$.

\begin{lemma}
\label{lem:meandiag}
There are block diagonal matrices $B_1,\ldots,B_N$, $N<\infty$, with blocks of size at most $2^d$, such that $B_k$ is an incidence matrix of a graph $(V,E_k),E_k\subset E,k=1,\ldots,N$ and
\[
\frac{1}{32}\ind_{{E}}\le\frac{1}{N}\sum_{k=1}^N B_k.
\]
\end{lemma}

\begin{proof}
We set $N=n$ and as $B_k$, $k=1,\ldots, n$ we take the 
adjacency matrix of the subgraph $\bigcup_{l\le s}(I_l,(I_l\times I_l)\cap E)\subset G$, with coordinates shifted cyclically $k$ times. In the whole proof, if we write `+', we mean addition $\mod n$. In particular, here for $X\subset\mathbb{N}$ and $y\in\mathbb{Z}$ we write $X+y=\{x+y \mod n:x\in X\}$. 

We start with deriving the following estimate
\begin{equation}
\label{eq:mean1}
A:=\frac{1}{n}\sum_{k=1}^n \ind_{((D_l+k)\times(D_l+k))\cap{E}}\ge\frac{m}{2n}\ind_{{E}}\quad\ \mathrm{for\ any}\ l\in[n].
\end{equation}
To this aim  observe first that $D_l+k=D_{l+k}$. Hence we have
\[
A=\frac{1}{n}\sum_{k=1}^n \ind_{(D_{l+k}\times D_{l+k})\cap{E}}=\frac{1}{n}\sum_{k=1}^n \ind_{(D_k\times D_k)\cap{E}}.
\]
Let us fix $r\in[d], i\in [n]$. Let $k\in [n]$ satisfy $i\in D_k$. By the definition of the lower cubes, there is $J\subseteq[d]$ such that $i=k-\sum_{j\in J} p_j$. If $r\in J$, then $i+p_r=k-\sum_{j\in J\setminus\{r\}}p_j\in D_k$, so $(i,i+p_r)\in D_k\times D_k$. 
Otherwise $i=k+p_r-\sum_{j\in J\cup\{r\}}p_j$ and $i+p_r=k+p_r-\sum_{j\in J}p_j$, so $(i,i+p_r)\in D_{k+p_r}\times D_{k+p_r}$.
Since $|\{k\in [n]\colon i\in D_k\}|=m$ we get
\[
A=\frac{1}{2n}\sum_{k=1}^n (\ind_{(D_k\times D_k)\cap{E}}+\ind_{(D_{k+p_r}\times D_{k+p_r})\cap{E}})
\geq \frac{m}{2n}\ind_{(i,i+p_r)}.
\]
In a similar way we show that $A\geq \frac{m}{2n}\ind_{(i,i-p_r)}$ for any $r\in[d], i\in [n]$ and \eqref{eq:mean1} easily follows.

Since $i\in J+k$ if and only if $k\in i-J$ and $|i-J|=|J|$ we have
\begin{equation}
\label{eq:mean2}
\frac{1}{n}\sum_{k=1}^n \ind_{((J+k)\times[n])\cap{E}}=\frac{|J|}{n}\ind_{{E}}\quad\mathrm{for\ any}\ J\subset[n].
\end{equation}

Now fix $l\in[n]$ and $I\subset D_l$ such that $|D_l\setminus I|\le cm$. Then
\[
\ind_{((I+k)\times(I+k))\cap{E}}\ge\ind_{((D_l+k)\times (D_l+k))\cap{E}}-\ind_{((D_l\setminus I+k)\times[n])\cap{E}}-\ind_{([n]\times(D_l\setminus I+k))\cap{E}},
\]
hence  \eqref{eq:mean1} and \eqref{eq:mean2} imply that
\begin{equation}
\label{eq:mean3}
\frac{1}{n}\sum_{k=1}^n \ind_{((I+k)\times(I+k))\cap{E}}\ge\frac{m(1-4c)}{2n}\ind_{{E}}.
\end{equation}
Let us take $k_1,\ldots,k_s$ as in Lemma \ref{lem:goodsequence} and set
\[
B_k:=\sum_{l=1}^s \ind_{((I_l+k)\times(I_l+k))\cap{E}}.
\]
Then, since $|D_{k_l}\setminus I_l|\le\frac{1}{8}m$ and $s\ge\frac{n}{8m}$, estimate \eqref{eq:mean3} implies that
\[
\frac{1}{n}\sum_{k=1}^n B_k\ge s\frac{m(1-4/8)}{2n}\ind_{{E}}\ge\frac{1}{32}\ind_{{E}}.
\]
\end{proof}

In the next lemma we prove a useful norm estimate for block diagonal matrices with blocks of fixed size. We are going to use the Hadamard product: $C=A\cdot B$ if $c_{ij}=a_{ij}b_{ij}$. By $\varepsilon=(\varepsilon_{ij})$ we denote a matrix with independent symmetric $\pm 1$ entries (its size may change from line to line). Recall definition \eqref{eq:defAvep} of the
norm $\|A\|_{\varepsilon,p}$.

\begin{lemma}
\label{lem:blocknorm}
Assume that $A$ is a block diagonal matrix of size $n\times n$, $n\geq 2$, with matrices $A_1,\ldots,A_m$ on the diagonal. 
Then\\
$(i)$ for any $p\ge 1$,
\[
\|A\|_{\varepsilon,p}=\max_{l\le m}\|A_l\|_{\varepsilon,p},
\]
$(ii)$ for any $m\le n$,
\[
\Ex\|\varepsilon\cdot A\|\le C\max_{l\le m}\left(\Ex\|\varepsilon\cdot A_l\|+\|A_l\|_{\varepsilon,\log(n+1)}\right)
\le C\left(\max_{l\le m}\Ex\|\varepsilon\cdot A_l\|+\|A\|_{\varepsilon,\log(n+1)}\right).
\]
\end{lemma}

\begin{proof}
The proof of $(i)$ is an easy exercise. For $(ii)$ we set $p=\log(n+1)$ and observe that
\begin{align*}
\Ex\|\varepsilon\cdot A\|&=\Ex\max_{l\le m}\|\varepsilon\cdot A_l\|\le\left(\Ex\max_{l\le m}\|\varepsilon\cdot A_l\|^p\right)^{1/p}\le\left(\sum_{l=1}^m\Ex\|\varepsilon\cdot A_l\|^p\right)^{1/p}\\
&\le m^{1/p}\max_{l\le p}(\Ex\|\varepsilon\cdot A_l\|^p)^{1/p}\le e\max_{l\le m}(\Ex\|\varepsilon\cdot A_l\|^p)^{1/p}.
\end{align*}

A well known result for Bernoulli processes (cf. \cite{DMS} or more general bound \cite[Theorem 1.1]{LaSt}) states that
\[
\left(\Ex\sup_{t\in T} \left|\sum_i t_i\varepsilon_i\right|^p\right)^{1/p}\le
C\left(\Ex\sup_{t\in T}\left|\sum_i t_i\varepsilon_i\right|+\sup_{t\in T}\left(\Ex\left|\sum_i t_i\varepsilon_i\right|^p\right)^{1/p}\right)
\]
for some constant $C$. Therefore
\[
(\Ex\|\varepsilon\cdot A_l\|^p)^{1/p}\le
C\left(\Ex\|\varepsilon\cdot A_l\|+C\|A_l\|_{\varepsilon,\log(n+1)}\right)
\]
and the second part of the assertion easily follows.
\end{proof}

Estimation of the operator norm $\Ex\|\varepsilon\cdot A_l\|$ is basically as difficult as Theorem \ref{thm:twosidedbound} itself. A comparison between Bernoulli and Gaussian processes gives an upper bound of the form
\begin{equation}
\label{eq:comparison}
\Ex\sup_{t\in T}\left|\sum_i t_i\varepsilon_i\right|\le C\Ex\sup_{t\in T}\left|\sum_i t_i g_i\right|,
\end{equation}
where $g_i$ are independent standard Gaussian random variables. This bound is far from optimal in general. However, under some specific assumptions on the matrix $A_l$, it becomes a sufficient tool for our considerations. Recall that two-sided bound for $\Ex\|(a_{ij}g_{ij})\|$ was obtained in \cite{HLY}. In particular, it holds that if $(a_{ij})$ is an $n\times n$ matrix and $|a_{ij}|\le 1$, then
\begin{equation}
\label{eq:gaussian}
\Ex\|(a_{ij}g_{ij})\|\le C\left(\max_i \sqrt{\sum_j a_{ij}^2}+\max_j\sqrt{\sum_i a_{ij}^2}+\sqrt{\log(n)}\right).
\end{equation}

\begin{corollary}
\label{cor:boundbydegree}
Let $A=(a_{ij})_{i,j\le n}$ be a matrix of size $n\times n$ and $d\le n/2$ a natural number. Assume that\\
$(i)$ There are natural numbers $1\le p_1\le\ldots\le p_d\le n/2$ such that if $a_{ij}\neq 0$, then $i-j=\pm p_k \mod n$ for some $1\le k\le d$,\\
$(ii)$ $|a_{ij}|\le 1$ for any $i,j$.\\
Then
\[
\Ex\|\varepsilon\cdot A\|\le C(\sqrt{d}+\|A\|_{\varepsilon,\log(n+1)}).
\]
\end{corollary}

\begin{proof}
Assume that $B_k=(b_{ij}(k))$, $k=1,\ldots,N$ are the block diagonal matrices given by Lemma \ref{lem:meandiag}. Then we have
\begin{align*}
\Ex\|\varepsilon\cdot A\|
&=\Ex\|\varepsilon\cdot A\cdot\ind_E\|\le
\Ex\bigg\|\varepsilon\cdot A\cdot\frac{32}{N}\sum_{k=1}^NB_k\bigg\|
\le\frac{32}{N}\sum_{k=1}^N\Ex\|\varepsilon\cdot A\cdot B_k\|
\\
&\le 32\max_{k\le N}\Ex\|\varepsilon\cdot A\cdot B_k\|,
\end{align*}
where the first inequality follows by Lemma \ref{lem:meandiag}, since, by the contraction principle \cite[Theorem 4.4]{LT}, 
$\Ex\|(\ve_{ij}c_{ij})\|\leq \Ex\|(\ve_{ij}d_{ij})\|$ if $|c_{ij}|\leq |d_{ij}|$ for all $i,j$.

Each of the matrices $\varepsilon\cdot A\cdot B_k$ is block diagonal, with blocks of size at most $2^d$. Let $I_1(k),\ldots,I_m(k)\subset[n]$ be such that the diagonal blocks are of the form $(a_{ij}b_{ij}(k)\varepsilon_{ij})_{i,j\in I_l(k)}$. The blocks have coefficients of absolute value at most $1$ and in any row at most $2d$ of them are nonzero. Hence,  using \eqref{eq:comparison} and \eqref{eq:gaussian} applied to the matrix $(a_{ij}b_{ij}(k)\varepsilon_{ij})_{i,j\in I_l(k)}$
with $|I_l(k)|\leq 2^d$, we obtain
\[
\Ex\|(a_{ij}b_{ij}(k)\varepsilon_{ij})_{i,j\in I_l(k)}\|\le C\sqrt{d},\quad k=1,\ldots,N,\quad l=1,\ldots,m.
\]
Therefore, by Lemma \ref{lem:blocknorm},
\[
\Ex\|\varepsilon\cdot A\cdot B_k\|\le C(\sqrt{d}+\|A\cdot B_k\|_{\varepsilon,\log(n+1)}).
\]
Since $B_k$ has 0-1 entries, it holds that $\|A\cdot B_k\|_{\varepsilon,\log(n+1)}\le\|A\|_{\varepsilon,\log(n+1)}$ and we finish the proof.
\end{proof}

\noindent
\textbf{Remark.} If we fix $\delta\in(0,1)$, then the upper bound in Theorem \ref{thm:twosidedbound}, without $\log\log$ terms and with  a constant depending on $\delta$, follows from Corollary \ref{cor:boundbydegree} under the assumption that either $a_{ij}=0$ or $\delta\le|a_{ij}|\le 1$ for any $i,j$.

\medskip

Corollary \ref{cor:boundbydegree} not only proves Theorem \ref{thm:twosidedbound} in this special 0-1 (or a bit wider) case. In general situation, the proof relies on dividing the matrix $(a_{ij})$ into parts, where all nonzero coordinates differ at most by a constant factor. Then again Corollary \ref{cor:boundbydegree} provides the crucial estimate.

\begin{proof}[Proof of Theorem \ref{thm:twosidedbound}]
The lower bound can be deduced from Theorem \ref{thm:lower}. To this aim it is sufficient to show that
\[
\sup_{\|s\|_2,\|t\|_2\le 1}\left\|\sum_{i,j} a_{ij}\varepsilon_{ij}s_i t_j\right\|_{\log(n+1)}
\le C\min_{I\subset[n],|I|\le n/4}\sup_{\|s\|_2,\|t\|_2\le 1}
\left\|\sum_{i,j\not\in I}a_{ij}\varepsilon_{ij}s_i t_j\right\|_{\log(n/4+1)},
\]
since the term on the right hand side is trivially bounded by the second term on the right hand side of \eqref{eq:lower}. We denote by $\bar{s},\bar{t}$ unit vectors realizing the supremum in the Rademacher norm:
\[
\sup_{\|s\|_2,\|t\|_2\le 1}\left\|\sum_{i,j}a_{ij}\varepsilon_{ij}s_i t_j\right\|_{\log(n+1)}
=\left\|\sum_{i,j}a_{ij}\varepsilon_{ij}\bar{s}_i \bar{t}_j\right\|_{\log(n+1)}.
\]
Such pair $(\bar{s},\bar{t})$ is not unique. Observe that circulant matrices have the following shift invariance property
(equality is meant in law)
\begin{equation}
\label{eq:shiftinv}
\sum_{i,j}a_{ij}\varepsilon_{ij}s_i t_j\stackrel{d}=\sum_{i,j}a_{ij}\varepsilon_{ij}s_{i+k}t_{j+k},\quad k=1,\ldots,n;\quad s,t\in\mathbb{R}^n.
\end{equation}
Here and in the whole proof addition inside indices is $\mod n$.

Fix arbitrary subset $I\subset[n]$ with $|I|\le n/4$. For any pair $(i,j)\in[n]\times[n]$ we have
\[
|\{k:i+k\in I\vee j+k\in I\}|\le\sum_{l\in I}(|\{k:i+k=l\}|+|\{k:j+k=l\}|)=2|I|\le\frac{n}{2}.
\]
Hence $i,j\not\in I-k$ for at least $n/2$ distinct values of $k$. It follows that
\[
\big|a_{ij}s_i t_j\big|\le
\bigg|\frac{2}{n}\sum_{k=1}^n a_{ij}s_i t_j \ind_{\{i,j\not\in I-k\}}\bigg|
\]
and the following estimate holds:
\begin{align*}
\left\|\sum_{i,j}a_{ij}\varepsilon_{ij}\bar{s}_i\bar{t}_j\right\|_{\log(n+1)}
&\le\frac{2}{n}\left\|\sum_{k=1}^n\sum_{i,j}a_{ij}\varepsilon_{ij}\bar{s}_i\bar{t}_j 
\ind_{\{i,j\not\in I-k\}}\right\|_{\log(n+1)}
\\
&\le\frac{2}{n}\sum_{k=1}^n\left\|\sum_{i,j\not\in I-k}a_{ij}\varepsilon_{ij}\bar{s}_i\bar{t}_j\right\|_{\log(n+1)}
\\
&\le 2\max_{1\le k\le n}\left\|\sum_{i,j\not\in I-k}a_{ij}\varepsilon_{ij}\bar{s}_{i}\bar{t}_{j}\right\|_{\log(n+1)}
\\
&\le C\sup_{\|s\|_2,\|t\|_2\le 1}\left\|\sum_{i,j\not\in I}a_{ij}\varepsilon_{ij}s_i t_j\right\|_{\log(n/4+1)},
\end{align*}
where the last inequality follows  from \eqref{eq:shiftinv} and the Khintchine inequality. This proves the lower bound.

Now we will show the upper bound.
Since the problem is homogenous, we may assume that $\max_{i,j}|a_{ij}|=1$. Let $A^{(k)}=(a_{ij}^{(k)})$ for 
$0\leq k\leq k_0:=\lfloor\log\log(n+3)\rfloor$, where
\begin{align*}
a_{ij}^{(k)}&=b_{i-j}^{(k)}:=b_{i-j}\ind_{e^{-k}<|b_{i-j}|\le e^{-k+1}},\quad k=1,\ldots,k_0,\\
a_{ij}^{(0)}&=b_{i-j}^{(0)}:=b_{i-j}\ind_{|b_{i-j}|\le e^{-k_0}}.
\end{align*}
Clearly $A^{(k)}$ are $n\times n$ circulant matrices, $(a_{ij})=\sum_k A^{(k)}$ and for any $i,j$ there is at most one $k$ such that $a_{ij}^{(k)}\neq 0$.
Applying \eqref{eq:comparison} and \eqref{eq:gaussian} to the matrix $e^{k_0} A^{(0)}$ we obtain
\begin{align*}
\Ex\|\varepsilon\cdot A^{(0)}\|&\le C\left(\max_i\sqrt{\sum_j (a_{ij}^{(0)})^2}+\max_j\sqrt{\sum_i (a_{ij}^{(0)})^2}+e^{-k_0}\sqrt{\log(n)}\right)\\
&\le C\left(\max_i\sqrt{\sum_j a_{ij}^2}+\max_j\sqrt{\sum_i a_{ij}^2}\right)=2C\sqrt{\sum_j b_j^2}.
\end{align*}
For $k\ge 1$ we assume that $d_k$ is the degree of $A^{(k)}$. Then Corollary \ref{cor:boundbydegree} applied to the matrix $e^{k-1}A^{(k)}$ gives
\begin{align}
\nonumber
\Ex\|\varepsilon\cdot A^{(k)}\|&\le C(e^{-k+1}\sqrt{d_k}+\|A^{(k)}\|_{\varepsilon,\log(n+1)})\\
\label{eq:partestimate}
&\le C\left(e\sqrt{\sum_j (b_j^{(k)})^2}+\|A^{(k)}\|_{\varepsilon,\log(n+1)}\right).
\end{align}

It is not hard to see that $\|A^{(k)}\|_{\varepsilon,\log(n+1)}\le\|(a_{ij})\|_{\varepsilon,\log(n+1)}$. Moreover, by the Cauchy-Schwartz inequality, for any $i$ we have
\[
\sum_{k=1}^{\log\log(n+3)}\sqrt{\sum_j (b_j^{(k)})^2}\le\sqrt{\log\log(n+3)}\sqrt{\sum_j b_j^2}.
\]
Now the triangle inequality yields
\begin{align*}
\Ex\|(a_{ij}\varepsilon_{ij})\|&\le\sum_{k=0}^{\log\log(n+3)}\Ex\|\varepsilon\cdot A^{(k)}\|\\
&\le C\sqrt{\sum_j b_j^2}+C\sum_{k=1}^{\log\log(n+3)}\left(\sqrt{\sum_j (b_j^{(k)})^2}+\|A^{(k)}\|_{\varepsilon,\log(n+1)}\right)\\
&\le C\left(\sqrt{\log\log(n+3)}\sqrt{\sum_j b_j^2}+\log\log(n+3)\|(a_{ij})\|_{\varepsilon,\log(n+1)}\right).
\end{align*}
\end{proof}

\noindent
\textbf{Remark.} The crucial observation in the proof of Theorem \ref{thm:twosidedbound} was that for a (directed) circulant graph
$(V,E)$ of degree $d$ there exists matrices $B_1,\ldots,B_N$ such that $\ind_E\leq \frac{1}{N}\sum_{k=1}^N B_k$ and $B_k$ are 
adjacency matrices of graphs with components of cardinality at most exponential in $d$. We do not know how broad is the class of graphs of degree $d$ with such property.

\section{Estimates for Rademacher norms}
\label{sec:radnorms}

In this section we will show estimates for the quantity
\[
\|(a_{ij})\|_{\ve,p}=\sup_{\|s\|_2,\|t\|_2\leq 1}\Bigg\|\sum_{i,j}a_{ij}\ve_{ij}s_it_j\Bigg\|_p
\] 
and apply them in few concrete situations.

We begin with the proof of
Proposition  \ref{prop:szacrad}, which gives two-sided bound for $\|A\|_{\ve,p}$
in the case of $0$-$1$ matrices.

\begin{proof}[Proof of Proposition \ref{prop:szacrad}]
To get the lower estimate let us fix $I\subset E$ of cardinality at most $p$. Then for $s,t\in \er^n$ we have
\begin{align*}
\Bigg\|\sum_{(i,j)\in I}\ve_{ij}s_it_j\Bigg\|_p
&\geq \sum_{(i,j)\in I}|s_i||t_j|\Prob\big(\forall_{(i,j)\in I}\ \ve_{ij}=\mathrm{sgn}(t_is_j)\big)^{1/p}
\geq \frac{1}{2}\sum_{(i,j)\in I}|s_i||t_j|.
\end{align*}
Taking the supremum over $s,t\in S^{n-1}$ we get 
$\|\ind_E\|_{\ve,p}\geq \frac{1}{2}\|\ind_{I}\|$.

To establish the reverse estimate define
\[
M
:=\max_{I\subset E, |I|\leq p}\|(\ind_{\{(i,j)\in I\}})\|.
\]
Bounding the operator norm of a matrix by its Hilbert-Schmidt norm we get $\|\ind_I\|\leq \sqrt{|I|}$, so that $M\leq \sqrt{p}$.
We will also assume that $p\geq 2$ is an integer (since we may change $p$ to $\max\{2,\lceil p\rceil\}\leq 2p$ -- observe that the RHS of \eqref{eq:estadj} is sublinear with respect to $p$).

To show the upper bound in \eqref{eq:estadj} it is enough to prove that 
\begin{equation}
\label{eq:estcut}
\sum_{(i,j)\in E}\min\Big\{s_i^2t_j^2,\frac{M^2}{p^2}\Big\}\leq C_1\frac{M^2}{p}\quad \mbox{ for } \|s\|_2,\|t\|_2\leq 1.
\end{equation}

Indeed, suppose that \eqref{eq:estcut} holds. Let us fix $s,t$ with $\|s\|_2,\|t\|_2\leq 1$ and define
$I:=\{(i,j)\in E\colon\ |s_it_j|\geq M/p\}$. Then by \eqref{eq:estcut} we have $|I|\leq C_1p$, so we may decompose $I$ into sets $I_1,I_2,\ldots I_{\lceil C_1\rceil}$ of cardinality at most $p$ each and get
\[
\Bigg\|\sum_{(i,j)\in I}\ve_{ij}s_it_j\Bigg\|_p\leq \sum_{(i,j)\in I}|s_it_j|\leq 
\sum_{l=1}^{\lceil C_1\rceil}\|\ind_{I_l}\|
\leq \lceil C_1\rceil M.
\]
On the other hand the Khintchine inequality yields
\[
\Bigg\|\sum_{(i,j)\in E\setminus I}\ve_{ij}s_it_j\Bigg\|_p\leq 
\sqrt{p}\Bigg(\sum_{(i,j)\in E\setminus I}s_i^2t_j^2\Bigg)^{1/2}
\leq \sqrt{p}\Bigg(\sum_{(i,j)\in E}\min\Big\{s_i^2t_j^2,\frac{M^2}{p^2}\Big\}\Bigg)^{1/2}\leq \sqrt{C_1}M.
\]
Thus by the triangle inequality,
\[
\sup_{\|s\|_2,\|t\|_2\leq 1}\Bigg\|\sum_{(i,j)\in E}\ve_{ij}s_it_j\Bigg\|_p\leq (\lceil C_1\rceil+\sqrt{C_1})M.
\]

Let us now make two simple observations on the cardinality of intersection of the set $E$ with rectangles.
To make the notation more concise set $r:=p^2/M^2\geq p$.

i) Every rectangle $R=J_1\times J_2$ such that $|J_1||J_2|< r$ contains at most
$(|J_1||J_2|)^{1/2}M<p$ edges from $E$.
Indeed, if we take $s:=|J_1|^{-1/2}\ind_{J_1}$ and $t:=|J_2|^{-1/2}\ind_{J_2}$ and take $I\subset E\cap R$ such
that $|I|=\min\{p,|E\cap R|\}$ we get
\[
M\geq \|(\ind_{\{(i,j)\in I\}})\|\geq \sum_{(i,j)\in I}|t_is_j|=(|J_1||J_2|)^{-1/2}|I|.
\]
This shows that $|I|\leq (|J_1||J_2|)^{1/2}M<r^{1/2}M=p$ so $|I|=|E\cap R|\leq (|J_1||J_2|)^{1/2}M$.

ii) Every rectangle $R=J_1\times J_2\subset V\times V$ such that $|J_1||J_2|\geq r$ contains at most
$8p|J_1||J_2|/r=8M^2|J_1||J_2|/p$ edges from $E$. Indeed we may decompose $R$ into sum of $k\leq 8|J_1||J_2|/r$ rectangles of area smaller than $r$ and use i). (To see such decomposition we consider a few cases. Case $|J_1|=1$
is pretty obvious. If $|J_2|< r$ we find $2\leq l\leq |J_1|$ such that $|J_1||J_2|/l<r\leq |J_1||J_2|/(l-1)$ and
partition $J_1$ into at most $2l\leq 4(l-1)\leq 4|J_1||J_2|/r$ sets of cardinatlity at most $|J_1|/l$. Finally
if $|J_2|>r$ and $|J_1|\geq 2$, we may decompose $J_2$ into at most $2|J_2|/r$ sets of cardinality smaller than
$r$ and use the method of further decomposition described in the previous step).

Since it is only a matter of permutation of rows and columns of the matrix $\ind_E$ (recall that we do not assume any symmetry) to establish
\eqref{eq:estcut}
we may assume without loss of generality that $|s_1|\geq |s_2|\geq\ldots\geq |s_n|$ and $|t_1|\geq |t_2|\geq\ldots\geq |t_n|$.
We also put $t_k=s_k=0$ for $k>n$.
Let $D_k:=\{2^{k},2^{k}+1,\ldots,2^{k+1}-1\}$.
Then by the monotonicity assumption and above observations,
\begin{align}
\notag
\sum_{(i,j)\in E}&\min\Big\{s_i^2t_j^2,\frac{M^2}{p^2}\Big\}
\leq \sum_{k,l\geq 0}\min\Big\{s_{2^k}^2t_{2^l}^2,\frac{1}{r}\Big\}|E\cap (D_k\times D_l)|
\\
\label{eq:estcut2}
&\leq \sum_{k+l< \log_2 r}M2^{\frac{k+l}{2}}\min\Big\{s_{2^k}^2t_{2^l}^2,\frac{1}{r}\Big\}
+\sum_{k+l\geq \log_2 r}8\frac{M^2}{p}2^{k+l}\min\Big\{s_{2^k}^2t_{2^l}^2,\frac{1}{r}\Big\}.
\end{align}
Observe that by monotonicity of $(s_k)$ we have
\[
\sum_{k\ge 0}2^ks_{2^k}^2\leq 2\left(s_1^2+\sum_{k\ge 0}2^ks_{2^{k+1}}^2\right)\leq 2\sum_{k\ge 0}s_k^2\leq 2.
\]
In the same way it follows that $\sum_{k\ge 0}2^k t_{2^k}^2\leq 2$. Therefore
\begin{equation}
\label{eq:estcut3}
\sum_{k+l\geq \log_2 r}2^{k+l}\min\Big\{s_{2^k}^2t_{2^l}^2,\frac{1}{r}\Big\}
\leq \sum_{k,l\geq 0}2^{k+l}s_{2^k}^2t_{2^l}^2
=\Bigg(\sum_{k\geq 0}2^ks_{2^k}^2\Bigg)\Bigg(\sum_{l\geq 0}2^lt_{2^l}^2\Bigg)\leq  4.
\end{equation}

Now we will estimate the second term in \eqref{eq:estcut2}. To this aim fix  $x>0$ and
define
\begin{align*}
l(k)&=l_x(k):=\max\{l\geq 0\colon\ s_{2^k}^2t_{2^l}^2\geq x\}
\quad \mbox{for}\quad k\in A(x):=\{k\geq 0\colon\ s_{2^k}^2{t_1}^2\geq x\},
\\
k(l)&=k_x(l):=\max\{k\geq 0\colon l(k)=l\}\quad \mbox{for}\quad l\in B(x):=\{l(k)\colon\ k\in A(x)\}.
\end{align*}
Note that the function $k=k(l)$ is strictly decreasing on $B(x)$. We have
\begin{align*}
\sum_{k,l\geq 0}2^{\frac{k+l}{2}}\ind_{\{s_{2^k}^2t_{2^l}^2\geq x\}}
&\leq \sum_{k\in A(x)} 2^{\frac{k}{2}}\sum_{l=0}^{l(k)}2^{\frac{l}{2}}
\leq \sum_{k\in A(x)} 2^{\frac{k}{2}}\frac{\sqrt{2}}{\sqrt{2}-1}2^{\frac{l(k)}{2}}
\\
&=\frac{\sqrt{2}}{\sqrt{2}-1}\sum_{l\in B(x)}2^{\frac{l}{2}}\sum_{k\colon l(k)= l}2^{\frac{k}{2}}
\leq \Big(\frac{\sqrt{2}}{\sqrt{2}-1}\Big)^2\sum_{l\in B(x)}2^{\frac{k(l)+l}{2}}
\\
&\leq \Big(\frac{\sqrt{2}}{\sqrt{2}-1}\Big)^2\sum_{l\in B(x)}2^{\frac{k(l)+l}{2}}\frac{|s_{2^{k(l)}}t_{2^l}|}{\sqrt{x}}
\\
&\leq \frac{1}{\sqrt{x}}\Big(\frac{\sqrt{2}}{\sqrt{2}-1}\Big)^2
\Big(\sum_{l\in B(x)}2^{k(l)}s_{2^{k(l)}}^2\Big)^{1/2}\Big(\sum_{l\in B(x)}2^{l}s_{2^l}^2\Big)^{1/2}
\\
&\leq \frac{1}{\sqrt{x}}\Big(\frac{\sqrt{2}}{\sqrt{2}-1}\Big)^2
\Big(\sum_{k\geq 0}2^{k}s_{2^{k}}^2\Big)^{1/2}\Big(\sum_{l\geq 0}2^{l}s_{2^l}^2\Big)^{1/2}
\\
&\leq \frac{2}{\sqrt{x}}\Big(\frac{\sqrt{2}}{\sqrt{2}-1}\Big)^2.
\end{align*}
Finally observe that 
\[
\min\Big\{a,\frac{1}{r}\Big\}\leq \sum_{u=0}^\infty \frac{1}{2^u r}\ind_{\{a\geq (2^ur)^{-1}\}}\quad\mbox{for}\quad a\geq 0,
\]
so
\begin{align}
\notag
\sum_{k,l\geq 0}2^{\frac{k+l}{2}}\min\Big\{s_{2^k}^2t_{2^l}^2,\frac{1}{r}\Big\}
&\leq \sum_{u=0}^\infty \frac{1}{2^u r}\sum_{k,l\geq 0}2^{\frac{k+l}{2}}\ind_{\{s_{2^k}^2t_{2^l}^2\geq (2^ur)^{-1}\}}
\\
\notag
&\leq \sum_{u=0}^\infty \frac{1}{2^u r}2(2^ur)^{1/2}\Big(\frac{\sqrt{2}}{\sqrt{2}-1}\Big)^2
\\
\label{eq:estcut4}
&=2\Big(\frac{\sqrt{2}}{\sqrt{2}-1}\Big)^3r^{-1/2}=2\Big(\frac{\sqrt{2}}{\sqrt{2}-1}\Big)^3\frac{M}{p}.
\end{align}

Estimates \eqref{eq:estcut2}-\eqref{eq:estcut4} imply \eqref{eq:estcut}, which completes the proof.

\end{proof}

The proof of the upper bound in Proposition \ref{prop:szacrad} strongly relies on the assumption that $a_{ij}\in\{0,1\}$. In the general case we can prove the same lower bound, but the upper estimate that we provide is significantly weaker.

\begin{proof}[Proof of Proposition \ref{prop:szacradA}]
The lower estimate follows by the same argument as in Proposition \ref{prop:szacrad}.

To get the reverse bound set 
\[
M:=\max_{I\subset [n]\times[n], |I|\leq p}\|(|a_{ij}|\ind_{\{(i,j)\in I\}})\|.
\]
By the Khintchine inequality we have
\begin{align*}
\Big\|(a_{ij}\ind_{\{|a_{ij}|\leq M/\sqrt{p}\}})\Big\|_{\ve,p}
&\leq \sqrt{p}\sup_{\|s\|_2\leq 1,\|t\|_2\leq 1}\sqrt{\sum_{ij}a_{ij}^2\ind_{\{|a_{ij}|\leq M/\sqrt{p}\}}s_i^2t_j^2}
\\
&=\sqrt{p}\max_{ij}|a_{ij}|\ind_{\{|a_{ij}|\leq M/\sqrt{p}\}}\leq M,
\end{align*}
so it is enough to estimate $\|\tilde{A}\|_{\ve,p}$, where $\tilde{a}_{ij}=a_{ij}\ind_{\{|a_{ij}|> M/\sqrt{p}\}}$.
Observe that 
\[
M\geq \max_{J\subset [n],|J|\leq p}\max\Big\{\max_i\sqrt{\sum_{j\in J}a_{ij}^2},\max_{j}\sqrt{\sum_{i\in J}a_{ij}^2}\Big\},
\]
so in each row and column there are at most $p$ nonzero elements of matrix $\tilde{A}$ and the length of each row and column
of $\tilde{A}$ is at most $M$. Thus by Lemma \ref{lem:smallsupportrad} we have that
\[
\|\tilde{A}\|_{\ve,p}\leq 2M+\sup_{\substack{\|s\|_2\leq 1,|\supp(s)|\leq p,\\ \|t\|_2\leq 1, |\supp(t)|\leq p}}
\left\|\sum_{ij}{\tilde a_{ij}}\ve_{ij}s_it_j\right\|_p
\leq 2M+\sup_{\substack{\|s\|_2\leq 1,|\supp(s)|\leq p,\\ \|t\|_2\leq 1, |\supp(t)|\leq p}}
\left\|\sum_{ij}a_{ij}\ve_{ij}s_it_j\right\|_p.
\]
To estimate the latter quantity take vectors $s,t\in S^{n-1}$ with support at most $p$ and define
\[
I_k:=\{(i,j)\in [n]\times [n]\colon |a_{ij}s_it_j|\geq 2^{-k}\},\quad k_0:=\sup\{k\colon\ |I_k|\leq p\}.
\]
Then $|I_{k_0}|\leq p$ and
\[
\Bigg\|\sum_{(i,j)\in I_{k_0}}a_{ij}\ve_{ij}s_it_j\Bigg\|_p\leq \sum_{(i,j)\in I_{k_0}}|a_{ij}||s_i||t_j| 
\leq\|(|a_{ij}|\ind_{\{(i,j)\in I_{k_0}\}})\|\leq M.
\]
Taking $\tilde{I}\subset I_{k_0+1}$ of cardinality $\lfloor p\rfloor \geq  p/2$ we see that
\begin{equation}
\label{eq:M:p}
M\geq \sum_{i,j\in \tilde{I}}|a_{ij}s_it_j|\geq 2^{-1-k_0}\lfloor p\rfloor\geq 2^{-2-k_0}p.
\end{equation}
By the Khintchine inequality we have
\[
\Bigg\|\sum_{(i,j)\notin I_{k_0}}a_{ij}\ve_{ij}s_it_j\Bigg\|_p\leq \sqrt{p}\Bigg(\sum_{(i,j)\notin I_{k_0}}a_{ij}^2s_i^2t_j^2\Bigg)^{1/2}.
\]
Let $(i,j)\in [n]\times [n]$ be such that $2^{1-l}>|a_{ij}s_it_j|\geq 2^{-l}$ for some $l>k_0$. Then
\[
a_{ij}^2{s_i^2t_j^2}\leq 2^{2-2l}=3\sum_{k\geq l}2^{-2k}=3\sum_{k=k_0+1}^\infty 2^{-2k}
\ind_{\{(i,j)\in I_k\}},
\]
therefore
\[
\Bigg\|\sum_{(i,j)\notin I_{k_0}}a_{ij}\ve_{ij}s_it_j\Bigg\|_p
\leq \sqrt{p}\Bigg(\sum_{i,j}3\sum_{k=k_0+1}^\infty 2^{-2k}\ind_{(i,j)\in I_k}\Bigg)^{1/2}
=\sqrt{3p}\Bigg(\sum_{k=k_0+1}^\infty 2^{-2k}|I_k|\Bigg)^{1/2}.
\]
Obviously $|I_k|\leq |\supp(s)||\supp(t)|\leq p^2$, so  using \eqref{eq:M:p} we obtain
\[
p\sum_{k\geq k_0+\frac{1}{2}\log_2 p}2^{-2k}|I_k|\leq p^3\sum_{k\geq k_0+\frac{1}{2}\log_2 p}2^{-2k}
\leq p^22^{1-2k_0}\leq 2^5M^2.
\]

Let us take $k> k_0$, then we may partition $[n]$ into disjoint sets $J_1,\ldots,J_{n_k}$ such that
\[
p\leq |I_k\cap(J_l\times [n])|\leq p+|\supp(t)|\leq 2p,\ l=1,\ldots,n_k-1,\quad |I_k\cap(J_{n_k}\times [n])|\leq 2p.
\]
Observe that $p(n_k-1)\leq |I_k|$ so $n_k\leq |I_k|/p+1\leq 2|I_k|/p$. 
We have
\begin{align*}
2^{-k}|I_k|&\leq \sum_{(i,j)\in I_k}|a_{ij}s_it_j|
\leq\sum_{l=1}^{n_k}\sum_{(i,j)\in I_k\cap (J_l\times [n])}|a_{ij}s_it_j|
=\sum_{l=1}^{n_k}\sum_{i\in J_l}|s_i|\sum_{j\in[n]}\ind_{\{(i,j)\in I_k\}}|a_{ij}t_j|\\
& \leq\sum_{l=1}^{n_k}\|(s_i)_{i\in J_l}\|_2\|(|a_{ij}|\ind_{(i,j)\in I_k\cap(J_l\times[n])})\|
\leq 
2M\sqrt{n_k}\|s\|_2\leq 2\sqrt{2}M\sqrt{|I_k|/p}.
\end{align*}
Thus $p2^{-2k}|I_k|\leq 8M^2$ and
\[
p\sum_{k_0<k<k_0+\frac{1}{2}\log_2 p}2^{-2k}|I_k|\leq 4M^2\log_2 p.
\]

\end{proof}

Propositions \ref{prop:szacrad} and \ref{prop:szacradA} provide bounds for  $\|A\|_{\ve,p}$, however
they involve suprema of operator norms that are quite hard to estimate. In the sequel we will discuss more conrete estimates, concentrating on the case when $A$ is the adjacency matrix of $d$-dimensional hypercube or more general $d$-dimensional discrete tori. 

In the case when $A=\ind_E$ is the adjacency matrix of a  graph $G=(V,E)$ we
will denote  $\|\ind_E\|_{\ve,p}$ by $N_{\ve,p}(G)$, i.e.
\[
N_{\ve,p}(G):=\sup_{\|s\|_2,\|t\|_2\leq 1}\Bigg\|\sum_{i,j\in V}\ind_{i\sim j}\ve_{ij}s_it_j\Bigg\|_p,
\] 
where $i\sim j$ means that $(i,j)\in E$.

The next simple lemma presents  general bounds that work well for small values of $p$.

\begin{lemma}  
\label{lem:estNp}
Let $G$ be a graph of maximum vertex degree $d$. Then $N_{\ve,p}(G)\leq \min\{d,\sqrt{p}\}$ for any $p\geq 1$ and 
$\sqrt{p/8}\leq N_{\ve,p}(G)\leq \sqrt{p}$ for  $1\leq p\leq d$.
\end{lemma}
 
\begin{proof}
We have 
\[
\Bigg\|\sum_{i,j\in V}\ind_{i\sim j}\ve_{ij}s_it_j\Bigg\|_p
\leq \sum_{i\sim j}|s_i||t_j|\leq \sum_{i\sim j}\frac{1}{2}(s_i^2+t_j^2)\leq \frac{d}{2}(\|s\|_2^2+\|t\|_2^2).
\]
Moreover, by the Khintchine inequality,
\[
\Bigg\|\sum_{i,j\in V}\ind_{i\sim j}\ve_{ij}s_it_j\Bigg\|_p 
\leq  \sqrt{p}\Bigg(\sum_{i\sim j}s_i^2t_j^2\Bigg)^{1/2}\leq \sqrt{p}\|s\|_2\|t\|_2.
\]
Hence the bound $N_{\ve,p}(G)\leq \min\{d,\sqrt{p}\}$ easily follows.

To obtain the last part of the assertion let us fix $1\leq p\leq d$, choose a vertex $i_0$ of degree $d$ and
let $J$ be a set consisting of $\lfloor p\rfloor$ of neighbours of $i_0$. Set $s:=\ind_{\{i_0\}}$ and
$t:=|J|^{-1/2}\ind_{J}$. Then
\begin{align*}
N_{\ve,p}(G)
&\geq
\Bigg\|\sum_{i,j\in V}\ind_{i\sim j}\ve_{ij}s_it_j\Bigg\|_p
=|J|^{-1/2}\Bigg\|\sum_{j\in J}\ve_{i_0j}\Bigg\|_p
\geq |J|^{1/2}\Prob(\ve_{i_0j}=1 \mbox{ for }j\in J )^{1/p}
\\
&\geq \frac{1}{2}|J|^{1/2}
\geq \frac{1}{2\sqrt{2}}p^{1/2}.
\end{align*}
\end{proof}

The next lemma gives bounds on $N_{\ve,p}(G)$ in terms of expansion and sparsity parameters  of $G$.

\begin{lemma}
\label{lem:estNp2}
Let $p\geq 1$ and $G=(V,E)$ be any graph. Then
\begin{equation}
\label{eq:estbyexp}
N_{\ve,p}\gtrsim\sup_{\emptyset\neq I,J\subset V}\frac{\min\{p,|E\cap(I\times J)|\}}{\sqrt{|I||J|}}.
\end{equation}
Moreover, if $G$ has maximum vertex degree $d$, then
\begin{equation}
\label{eq:estbysparsity}
N_{\ve,p}\lesssim \sqrt{d}\bigg(\sup_{\emptyset\neq I\subset V}\frac{\min\{p,|E\cap(I\times I)|\}}{|I|}\bigg)^{1/2}.
\end{equation}
\end{lemma}

\begin{proof}
To show the lower bound let us fix $\emptyset\neq I,J\subset V$ and choose 
$E_1\subset E\cap(I\times J)$ such
that $p\geq |E_1|\geq \frac{1}{2}\min\{p,|E\cap(I\times J)|\}$. Let $s=|I|^{-1/2}\ind_{I}$ and 
$t=|J|^{-1/2}\ind_{J}$. Then by Proposition \ref{prop:szacrad}
\[
N_{\ve,p}(G)\gtrsim \|(\ind_{\{(i,j)\in E_1\}})\|
\geq \sum_{(i,j)\in E_1}s_it_j=\frac{|E_1|}{\sqrt{|I||J|}}\geq 
\frac{1}{2}\frac{\min\{p,|E\cap(I\times J)|\}}{\sqrt{|I||J|}}
\]
and we get \eqref{eq:estbyexp}.

Now we will prove the upper bound. To this aim define
\[
M:=\sup_{\emptyset\neq I\subset V}\frac{\min\{p,|E\cap(I\times I)|\}}{|I|}.
\]
By Proposition \ref{prop:szacrad} to establish \eqref{eq:estbysparsity} it is enough to show that
\begin{equation}
\label{eq:boundE1}
\sup_{E_1\subset E,|E_1|\leq p}\sup_{\|s\|_2,\|t\|_2\leq 1}\sum_{(i,j)\in E_1}|s_it_j|\lesssim \sqrt{dM}.
\end{equation}

Let us fix $E_1\subset E$ of cardinality at most $p$ and vectors $s,t$ with $\|s\|_2\leq 1$ and
$\|t\|_2\leq 1$.
Set
\[
I_k(u):=\{i\in V\colon\ 2^{-k}<|u_i|\leq 2^{1-k}\},\quad u\in\{s,t\},\ k=1,2,\ldots.
\]
Then
\[
\sum_{(i,j)\in E_1}|s_it_j|
\leq \sum_{k,l=1}^{\infty}2^{1-k}2^{1-l}|E_1\cap (I_k(s)\times I_l(t))|.
\]
Observe that
\[
|E_1\cap (I_k(s)\times I_l(t))|\leq |E_1\cap ((I_k(s)\cup I_l(t))\times (I_k(s)\cup I_l(t)))|\leq
M|I_k(s)\cup I_l(t)|.
\]
We also have 
\[
|E_1\cap (I_k(s)\times I_l(t))|\leq |E\cap (I_k(s)\times I_l(t))|\leq \min\{d|I_k(s)|,d|I_l(t)|\}.
\]
Therefore
\[
\sum_{(i,j)\in E_1}|s_it_j|\leq 
4\sum_{k,l=1}^\infty2^{-k-l}(\min\{M|I_k(s)|,d|I_l(t)|\}+\min\{d|I_k(s)|,M|I_l(t)|\}).
\]
We have
\begin{align*}
\sum_{k,l=1}^{\infty}&2^{-k-l}\min\{M|I_k(s)|,d|I_l(t) |\}
\\
&\leq
M\sum_{k=1}^\infty 2^{-k}|I_k(s)|\sum_{l\colon 2^{-l}\leq 2^{-k}\sqrt{d/M}}2^{-l}
+d \sum_{l=1}^\infty 2^{-l}|I_l(t)|\sum_{k\colon 2^{-k}< 2^{-l}\sqrt{M/d}}2^{-k}
\\
&\leq M\sum_{k=1}^\infty 2^{1-2k}|I_k(s)|\sqrt{d/M}+d\sum_{l=1}^\infty 2^{1-2l}|I_l(t)|\sqrt{M/d}
\leq 2\sqrt{dM}(\|s\|_2^2+\|t\|_2^2)
\leq 4\sqrt{dM}.
\end{align*}
In a similar way we show that
\[
\sum_{k,l=1}^{\infty}2^{-k-l}\min\{M|I_l(t)|,d|I_k(s)|\}\leq 4\sqrt{dM}
\]
and \eqref{eq:boundE1} easily follows.

\end{proof}

The next proposition shows how to apply previous general bounds to the case of the Hamming hypercube.
}

\begin{proposition}
\label{prop:z2dnorms}
Let $G=(V,E)$ be the Hamming hypercube $V=\zet_2^d=\{0,1\}^d$. Then $N_{\ve,p}(G)\sim \sqrt{p}$ for $2\leq p\leq d$,
$N_{\ve,p}(G)\sim d$ for $p\geq d2^d$ and
\begin{equation}
\label{eq:estcube}
\sqrt{d\frac{\ln p}{\ln (\frac{ed}{\ln p})}}
\lesssim N_{\ve,p}(G)
\lesssim \sqrt{d\ln p}\quad \mbox{for }\max\{d,2\}\leq p\leq d2^d.
\end{equation}
\end{proposition}

\begin{proof}
If $p\geq |E|=d2^d$ then \eqref{eq:estbyexp} applied with $I=J=V$ shows that
$N_{\ve,p}(E)\gtrsim |E|/|V|=d$. Since $N_{\ve,p}(G)\leq d$ by Lemma \ref{lem:estNp}
we get the first part of the assertion.

To see the lower estimate in \eqref{eq:estcube} define, for $0\leq l\leq n$, $V_l$ as the set of all vertices from $V$ with exactly
$l$ coordinates equal to $1$. Then $a_l:=|V_l|=\binom{d}{l}$. 
There are exactly $ka_k$ edges in $V_k\times V_{k-1}$, so for $1\leq k\leq d/2$ such that $p\geq ka_k$ we have by \eqref{eq:estbyexp}
\[
N_{\ve,p}(G)\gtrsim  ka_{k-1}^{-1/2}a_k^{1/2} =\sqrt{(d-k+1)k}\sim\sqrt{dk}.
\]
We have $ka_k=k\binom{d}{k}\leq (2ed/k)^k$, hence the condition
$p\geq ka_k$ holds (recall that $d\leq p\leq d2^d$) for 
$k$ of the order $\ln p/\ln(ed/\ln p)$ and the lower bound in \eqref{eq:estcube} follows.

To get the upper bound we will use the second part of Lemma \ref{lem:estNp2}.
Harper's edge-isoperimetric inequality on the hypercube \cite{Ha} states that for any set $I\subset V$ with cardinality
at most $2^k$, $|E\cap (I\times I)|\leq k2^k$. So for $p\geq 2$, 
\[
\sup_{\emptyset\neq I\subset V}\frac{\min\{p,|E\cap(I\times I)|\}}{|I|}
\leq 1+\sup_{\emptyset\neq I\subset V,|I|\leq p}\frac{|E\cap(I\times I)|}{|I|}\lesssim \ln p
\]
and \eqref{eq:estbysparsity} yields the upper bound in \eqref{eq:estcube}.
\end{proof}

We may easily extend bounds from the previous proposition to the case $\zet_m^d$, due to the following simple lemma.

\begin{lemma}
\label{lem:zetmdnorms}
For any $p\geq 1$ and positive integers $k,d$ we have 
\[
N_{\ve,p}(\zet_2^d)\leq N_{\ve,p}(\zet_{2k}^d)\leq 2N_{\ve,p}(\zet_{2}^d),\quad
N_{\ve,p}(\zet_2^d)\leq N_{\ve,p}(\zet_{2k+1}^d)\leq 3N_{\ve,p}(\zet_{2}^d).
\]
\end{lemma}

\begin{proof}
Lower bounds follow from the fact that $\zet_2^d$ is a subgraph of $\zet_m^d$ for any $m\geq 2$. 

To show the upper bound, we first consider the case $\zet_{2k}^d$. Then there are two partitions of
$\zet_{2k}=\{1,\ldots,2k\}$ into pairs:  $I_{1,l}:=\{2l-1,2l\}$, $1\leq l\leq k$ and   
$I_{2,l}=\{2l,2l+1\}$,  $1\leq l\leq k$ (we identify $m+1$ with $1$). For a fixed
$\bfi=(i_1,i_2,\ldots,i_d)\in \{1,2\}^d$ we may treat 
$I_{\bfi,\bfl}=I_{i_1,l_1}\times I_{i_2,l_2}\times\cdots\times I_{i_d,l_d}$, $\bfl=(l_1,\ldots,l_d)\in \{1,\ldots,k\}^d$
as disjoint subgraphs of $\zet_{2k}^d$ isomorphic to $\zet_2^d$. 
Let $E_{\bfi }$ denote edges of $\zet_{2k}^d$ joining vertices from 
$I_{\bfi,\bfl}$ for some $\bfl\in \{1,\ldots,k\}^d$ and let $A_{\bfi}$ be the adjacency matrix of
$(\{1,\ldots,m\}^d,E_{\bfi})$. Then $A_\bfi$ is a block-diagonal matrix with
$k^d$-blocks such that each block $A_{\bfi,\bfl}$ is isomorphic to the adjacency matrix of $\zet_2^d$.
Thus part (i) of Lemma \ref{lem:blocknorm} yields 
\[
\|A_\bfi\|_{\ve,p}=\max_{\bfl} \|A_{\bfi,\bfl}\|_{\ve,p}=N_{\ve,p}(\zet_2^d).
\]
Observe that every edge of $\zet_{2k}^d$ belongs to exactly $2^{d-1}$ sets
$E_{\bfi}$ (if vertices of this edge differ at the cooordinate $k$, there is only one way to choose $i_k$ and all
other $i_j$  $j\neq k$ may be chosen in an arbitrary way). Hence we have $A=2^{1-d}\sum_\bfi A_\bfi$ (where $A$ is the adjacency
matrix of $\zet_{2k}^d$) and
\[
N_{\ve,p}(\zet_{2k}^d)=\|A\|_{\ve,p}\leq 2^{1-d}\sum_\bfi\|A_\bfi\|_{\ve,p}
=2^{1-d}2^dN_{\ve,p}(\zet_2^d)=2N_{\ve,p}(\zet_2^d).
\]

In the case $\zet_{2k+1}^d$ we proceed in a similar way, but we consider 3 types of disjoint families of pairs
$\mathcal{I}_1=\{\{2l-1,2l\}, \ 1\leq l\leq k\}$, $\mathcal{I}_2=\{\{2l,2l+1\}, \ 1\leq l\leq k\}$
and  $\mathcal{I}_3=\{\{2k+1,1\}\}$.

\end{proof}

\begin{corollary}
\label{cor:tor}
Let $d,m\geq 2$ and $G$ be the torus $\zet_m^d$. Then
\begin{equation}
\label{eq:estrandtorus}
\Ex\|(\ind_{i\sim j}\ve_{ij})_{i,j\in \zet_m^d}\|\sim \sqrt{d}+N_{\ve,d\log m}(\zet_2^d).
\end{equation}
In particular for $2\leq m\leq e^d$ we have 
\[
 \sqrt{d\log(d\log m)/\log(ed/\log(d\log m))}\leq \Ex\|(\ind_{i\sim j}\ve_{ij})_{i,j\in \zet_m^d}\| \lesssim \sqrt{d\log(d\log m)}.
\]
\end{corollary}

\begin{proof}
We may identify $\{1,\ldots,m\}^d$ with $\{1,\ldots,m^d\}$ by $(i_1,\ldots,i_d)\rightarrow \sum_{k=1}^d i_km^{k-1}$.
After such identification the graph $\zet_m^d$ becomes a subgraph of a circulant graph on $n=m^d$ vertices defined via $4d$ bands corresponding to
$\pm m^{k-1}$, $\pm (m-1)m^{k-1}$, $k=1,\ldots,d$. Thus by Corollary \ref{cor:boundbydegree} we have
\[
\Ex\|(\ind_{i\sim j}\ve_{ij})_{i,j\in \zet_m^d}\|\lesssim \sqrt{4d}+N_{\ve,d\log m}(\zet_m^d)\sim \sqrt{d}+N_{\ve,d\log m}(\zet_2^d),
\]
where the last equivalence holds by Lemma \ref{lem:zetmdnorms}.

For $m=2,3$ by Lemma \ref{lem:estNp} we have $N_{d\log m}(\zet_2^d)\sim \sqrt{d}\lesssim  \Ex\|(\ind_{i\sim j}\ve_{ij})_{i,j\in \zet_m^d}\|$ which shows the lower bound in \eqref{eq:estrandtorus}
for such $m$. For $m\geq 4$ observe that $\zet_m^d$ contains at least 
$(\lfloor m/2\rfloor)^d\geq m^{d/3}:=k$ disjoint copies of $\zet_2^d$. Thus Theorem \ref{thm:lower}
yields that  for such $m$
\[
 \Ex\|(\ind_{i\sim j}\ve_{ij})_{i,j\in \zet_m^d}\|\gtrsim \sqrt{d}+N_{\ve,\log (k+1)}(\zet_2^d)\sim \sqrt{d}+N_{\ve,d\log m}(\zet_2^d). 
\]

The last part of the assertion follows by \eqref{eq:estrandtorus} and Proposition \ref{prop:z2dnorms}.
\end{proof}

\section{Extensions}
\label{sec:ext}

In this section we will consider a more general class of random matrices with entries satisfying the condition
\begin{equation}
\label{eq:ass}
\|X_{ij}\|_{2p}\leq \alpha \|X_{ij}\|_p \quad \mbox{ for all $i,j\leq n$ and $p\geq 1$},
\end{equation}
where $\alpha\geq 1$ is a fixed constant. We may generalize to this case the bound from Theorem \ref{thm:lower}.

\begin{theorem}
\label{thm:lowergen}
Let $(X_{ij})_{i,j\leq n}$ be independent, mean zero r.v's satisfying condition \eqref{eq:ass}. Then
\begin{align}
\notag
\Ex\|(X_{ij})\|&\gtrsim_\alpha 
\max_{1\leq i\leq n}\left(\sum_{j=1}^n \Ex X_{ij}^2\right)^{1/2}
+\max_{1\leq j\leq n}\left(\sum_{i=1}^n \Ex X_{ij}^2\right)^{1/2}
\\
\label{eq:lowergen}
&\phantom{\gtrsim_\alpha }
+\max_{1\leq k\leq n}\min_{I\subset [n],|I|\leq k}
\sup_{\|s\|_2,\|t\|_2\leq 1}\left\|\sum_{i,j\notin I}X_{ij}s_it_j\right\|_{\log (k+1)}.
\end{align}
\end{theorem}

In order to show this result we need the following generalization of Lemma \ref{lem:smallsupportrad}.

\begin{lemma}
\label{lem:smallsupport}
Let $\beta:=\max\{1/2, \log_2\alpha\}$, then for any $p\geq 2$,
\begin{align*}
&\sup_{\|s\|_2,\|t\|_2\leq 1}\left\|\sum_{ij}X_{ij}s_it_j\right\|_p
\\
&\leq \sup_{\substack{\|s\|_2\leq 1,|\supp(s)|\leq p^{2\beta},\\ \|t\|_2\leq 1, |\supp(t)|\leq p^{2\beta}}}
\left\|\sum_{ij}X_{ij}s_it_j\right\|_p
+C(\alpha)\Bigg(\max_{i}\left(\sum_{j}\Ex X_{ij}^2\right)^{1/2}+\max_{j}\left(\sum_{j}\Ex X_{ij}^2\right)^{1/2}\Bigg).
\end{align*}
\end{lemma}

\begin{proof}
We proceed in the similar way as in the proof of Lemma \ref{lem:smallsupportrad}. We change $p^{-1}$ to $p^{-2\beta}$ in the definition of sets $I_k(z)$ and use
the following extension of the Khintchine inequality (cf. Lemma 4.1 in \cite{LaSt})
\begin{equation}
\label{eq:Khingen}
\left\|\sum_{i,j}u_{ij}X_{ij}\right\|_p\leq C(\alpha)\Big(\frac{p}{q}\Big)^{\beta}\left\|\sum_{i,j}u_{ij}X_{ij}\right\|_q\quad p\geq q\geq 2.
\end{equation}
\end{proof}

\begin{proof}[Proof of Theorem \ref{thm:lowergen}]
The proof of Theorem \ref{thm:lower} works here with the following modifications:
\begin{itemize}
\item The constants depend on $\alpha$;
\item We use the generalized Khintchine-Kahane inequality: $\Ex\|(X_{ij})\|\sim_\alpha (\Ex\|(X_{ij})\|^2)^{1/2}$ (cf. \cite{LaSt});
\item Instead of the Khintchine inequality we apply \eqref{eq:Khingen};
\item To get $\Prob(S_l\geq c'(\alpha)\|S_l\|_{\log(k+1)})\geq \frac{c'(\alpha)}{\sqrt{k}}$ we use the Paley-Zygmund inequality and 
\eqref{eq:Khingen} as in the proof of (4.6) in \cite{LaSt}.
\end{itemize}
\end{proof}

Theorem \ref{thm:lowergen} motivates the following generalization of Conjecture \ref{conj:rademacher}.

\begin{conjecture}
\label{conj:main}
Let $(X_{ij})_{i,j\leq n}$ be independent, mean zero r.v's satisfying condition \eqref{eq:ass}. Then
\begin{align*}
\Ex\|(X_{ij})\|
\sim_{\alpha}
&\max_{i}\left(\sum_{j} \Ex X_{ij}^2\right)^{1/2}+\max_{j}\left(\sum_{i} \Ex X_{ij}^2\right)^{1/2}
\\
&+\max_{1\leq k\leq n}\min_{I\subset [n],|I|\leq k}
\sup_{\|s\|_2,\|t\|_2\leq 1}\left\|\sum_{i,j\notin I}X_{ij}s_it_j\right\|_{\log (k+1)}.
\end{align*}
\end{conjecture}

The operator norm is trivially bigger than maximum length of columns/rows.
In \cite{HLY} it was shown that in the case when $X_{ij}$ are mixtures of Gaussian r'v's then this bound may be reversed in expectation:
\[
\Ex\|(X_{ij})\|\sim \Ex\max_{i}\sqrt{\sum_j X_{ij}^2}+\Ex\max_{j}\sqrt{\sum_i X_{ij}^2}.
\]
The proposition below implies in particular that Conjecture \ref{conj:main} holds for mixtures of Gaussian variables.

\begin{proposition}
\label{prop:uppermaxrows}
Let $(X_{ij})_{i,j\leq n}$ be independent, mean zero r.v's satisfying condition \eqref{eq:ass}. Then
\begin{align*}
\Ex\max_{i}\sqrt{\sum_j X_{ij}^2}
&\lesssim_{\alpha}
\max_{i}\left(\sum_{j} \Ex X_{ij}^2\right)^{1/2}
+\max_{1\leq k\leq n}\min_{I\subset [n],|I|\leq k}
\sup_{\|t\|_2\leq 1}\max_{i\notin I}\left\|\sum_{j}X_{ij}t_j\right\|_{\log (k+1)}
\\
&
\lesssim_{\alpha}\max_{i}\left(\sum_{j} \Ex X_{ij}^2\right)^{1/2}+\max_{j}\left(\sum_{i} \Ex X_{ij}^2\right)^{1/2}
\\
&
\phantom{aa}
+\max_{1\leq k\leq n}\min_{I\subset [n],|I|\leq k}\sup_{\|t\|_2\leq 1}\max_{i\notin I}
\left\|\sum_{j\notin I}X_{ij}t_j\right\|_{\log (k+1)}.
\end{align*}
\end{proposition}

\begin{proof}
Let 
\[
\gamma:=\max_{1\leq k\leq n}\min_{I\subset [n],|I|\leq k}
\sup_{\|t\|_2\leq 1}\max_{i\notin I}\left\|\sum_{j}X_{ij}t_j\right\|_{\log (k+1)}.
\]
Let us put rows/columns of $X$ in such a way that for $1\leq l\leq \log n$,
\[
\sup_{\|t\|_2\leq 1}\max_{i\geq 2^{l+1}}\left\|\sum_{j}X_{ij}t_j\right\|_{l}\leq \gamma.
\]
The weak-strong concentration result \cite{LaSt} shows that
\[
\max_{2^{l+1}\leq i\leq 2^{l+2}}\left\|\sqrt{\sum_j X_{ij}^2}\right\|_l\leq_{\alpha}
\max_{i}\left(\sum_{j} \Ex X_{ij}^2\right)^{1/2}
+\sup_{\|t\|_2\leq 1}\max_{2^{l+1}\leq i\leq 2^{l+2}}\left\|\sum_{j}X_{ij}t_j\right\|_{l}.
\]
Finally it is not hard to observe that
\[
\Ex\max_{i\geq 4}\sqrt{\sum_j X_{ij}^2}
\lesssim \max_{l\geq 1}\max_{2^{l+1}\leq i\leq 2^{l+2}}\left\|\sqrt{\sum_j X_{ij}^2}\right\|_l.
\]

To see the last estimate in the assertion let us fix $1\leq k\leq n$ and choose $I_k\subset [n]$ such that
$|I_k|\leq k$ and 
\[
\min_{I\subset [n],|I|\leq k}
\sup_{\|t\|_2\leq 1}\max_{i\notin I}\left\|\sum_{j\notin I}X_{ij}t_j\right\|_{\log (k+1)}
=\sup_{\|t\|_2\leq 1}\max_{i\notin I_k}\left\|\sum_{j\notin I_k}X_{ij}t_j\right\|_{\log (k+1)}.
\]
For any $j\in I_k$ let $I(j)$ be the subset of $[n]$ containing indices of $(k-1)$ largest elements of the sequence
$(\Ex X_{ij}^2)_{i\in [n]}$. Put $\tilde{I}_k=I_k\cup \bigcup_{j\in I_k}I(j)$. Then $|\tilde{I}_k|\leq k^2$ and
by \eqref{eq:Khingen}
\[
\gamma\leq \max_{1\leq k\leq n}\sup_{ \|t\|_2\leq 1}\max_{i\notin \tilde{I}_k}\left\|\sum_{j}X_{ij}t_j\right\|_{\log (k^2+1)}
\leq C(\alpha)\max_{1\leq k\leq n}\sup_{ \|t\|_2\leq 1}\max_{i\notin \tilde{I}_k}\left\|\sum_{j}X_{ij}t_j\right\|_{\log (k+1)}.
\]
We have 
\[
\sup_{\|t\|_2\leq 1}\max_{i\notin \tilde{I}_k}\left\|\sum_{j\notin I_k}X_{ij}t_j\right\|_{\log (k+1)}
\leq \sup_{\|t\|_2\leq 1}\max_{i\notin I_k}\left\|\sum_{j\notin I_k}X_{ij}t_j\right\|_{\log (k+1)}
\]
and, applying again \eqref{eq:Khingen}, 
\begin{align*}
\sup_{\|t\|_2\leq 1}\max_{i\notin \tilde{I}_k}\left\|\sum_{j\in I_k}X_{ij}t_j\right\|_{\log (k+1)}
&\leq C(\alpha)\log^\beta(k+1)\sup_{\|t\|_2\leq 1}\max_{i\notin \tilde{I}_k}\left\|\sum_{j\in I_k}X_{ij}t_j\right\|_{2}
\\
&=C(\alpha)\log^\beta(k+1)\max_{i\notin \tilde{I}_k}\max_{j\in I_k}(\Ex X_{ij}^2)^{1/2}
\\
&\leq C(\alpha)\log^\beta(k+1)\max_{j\in I_k}\max_{i\notin I(j)}(\Ex X_{ij}^2)^{1/2}
\\
&\leq C(\alpha)\log^\beta(k+1)\frac{1}{\sqrt{k}}\max_{j\in I_k}\Big(\sum_i\Ex X_{ij}^2\Big)^{1/2}
\\
&\leq C'(\alpha)\max_j\Big(\sum_i\Ex X_{ij}^2\Big)^{1/2}.
\end{align*}

\end{proof}

\begin{rem}
The decomposition/permutation trick of \cite{HLY} shows that in order to show the upper part of Conjecture \ref{conj:main} it is enough to prove that 
\begin{align}
\notag
\Ex\|(X_{ij})_{i,j\leq n}\|
&\lesssim_\alpha \max_{i}\left(\sum_{j} \Ex X_{ij}^2\right)^{1/2}+\max_{j}\left(\sum_{i} \Ex X_{ij}^2\right)^{1/2}
\\
\label{eq:reduction}
&\phantom{aaa}+
\sup_{\|s\|_2,\|t\|_2\leq 1}\left\|\sum_{i,j}X_{ij}s_it_j\right\|_{\log (n+1)}.
\end{align}
\end{rem}

\begin{proof}[Sketch of the proof]
The standard argument (cf. proof of \cite[Corollary 4.1]{HLY}) shows that we may consider symmetric matrices $X=(X_{ij})$, i.e. the case when $X_{ij}=X_{ji}$ and $(X_{ij})_{i\geq j}$ are independent mean zero r.v's satisfying condition \eqref{eq:ass}.
We assume that \eqref{eq:reduction} holds for any square submatrix of $X$ and we will show that
\[
\Ex\|(X_{ij})\|
\lesssim_{\alpha} a+\gamma,
\]
where
\[
a:=\max_{i}\left(\sum_{j} \Ex X_{ij}^2\right)^{1/2},\quad
\gamma:=\max_{k\geq 1}\min_{|I|\leq k}
\sup_{\|s\|_2,\|t\|_2\leq 1}\left\|\sum_{i,j\notin I}X_{ij}s_it_j\right\|_{\log (k+1)}.
\]
Observe that for $p\geq 2$, $\|X_{ij}\|_p\leq \alpha p^{\log_2\alpha}\|X_{ij}\|_2\leq \alpha p^\beta\|X_{ij}\|_2$, where $\beta=\max\{1/2,\log_2\alpha\}$. Thus the proof of the 
upper bound in \cite[Theorem 4.4]{HLY} shows that 
\begin{equation}
\label{eq:th44}
\Ex\|(X_{ij})_{i,j}\|\lesssim_\alpha a+\max_{i}\log^\beta(i)\max_{j}\|X_{ij}\|_2.
\end{equation}
Now we  construct a permutation $(i_1,\ldots,i_n)$ of indices $\{1,\ldots,n\}$ in a similar
way as in the proof of \cite[Theorem 3.9]{HLY}. We set $N_0=1$ and $N_k=2^{2^k}$ for $k\geq 1$.
We choose $I_1=\{i_1,\ldots,i_{N_1}\}$ in such a way that
\[
\sup_{\|s\|_2,\|t\|_2\leq 1}\left\|\sum_{i,j\notin I_1}X_{ij}s_it_j\right\|_{\log (N_1+1)}
\leq \gamma.
\] 
Suppose that we have selected $I_k=\{i_1,\ldots,i_{N_k}\}$. Then first among the remaining indices we choose $N_{k}N_{k-1}$ indices $i$ that contain the $N_{k-1}$ largest 
$L_2$-norms $\|X_{ij}\|_2$ of each column $j\in I_{k}$. Finally among remaining indices we choose
$N_{k+1}-N_k-N_{k}N_{k-1}\geq N_{k}$ indices $i$ in such a way that  
$I_{k+1}=\{i_1,\ldots,i_{N_{k+1}}\}$ satisfies
\[
\sup_{\|s\|_2,\|t\|_2\leq 1}\left\|\sum_{i,j\notin I_{k+1}}X_{ij}s_it_j\right\|_{\log (N_{k}+1)}
\leq \gamma.
\] 
The above construction implies in particular that
\begin{equation}
\label{eq:estoffdiag}
\|X_{ij}\|_2\leq aN_{k-1}^{-1/2}=a2^{-2^{k-2}}\quad \mbox{ when }j\leq N_k \mbox{ and }i\geq M_k:=N_k+N_kN_{k-1}.
\end{equation}
We set
\[
E_1:=[1,M_1]^2\cup\bigcup_{k\geq 1}[N_{2k},M_{2k+1}],\quad
E_2:=\bigcup_{k\geq 1}[N_{2k-1},M_{2k}]\setminus E_1,\quad E_3:=[1,n]^2\setminus(E_1\cup E_2)
\]
and write $X=U+V+W$, where 
\[
U_{ij}:=X_{ij}\ind_{\{(i,j)\in E_1\}},\quad V_{ij}:=X_{ij}\ind_{\{(i,j)\in E_2\}},
\quad W_{ij}:=X_{ij}\ind_{\{(i,j)\in E_3\}}.
\]
Matrix $U$ is block diagonal with the first block $U_1=(X_{ij})_{i,j\in J_1}$ of dimension 
$|J_1|=M_1$ and
blocks $U_k=(X_{ij})_{i,j\in J_k}$ for $k\geq 2$ of dimension $|J_k|=M_{2k-1}-N_{2k-2}+1$. We have
\[
(\Ex\|U_1\|^2)^{1/2}\leq \Big(\sum_{i,j\in [1,M_1]}\Ex X_{ij}^2\Big)^{1/2}\leq M_1^{1/2}a
\]
and for $k=2,3,\ldots$
\begin{align*}
\Big(\Ex\|U_k\|^{2^k}\Big)^{2^{-k}}
&\sim_\alpha \Big(\Ex\|U_k\|+\sup_{\|s\|_2,\|t\|_2\leq 1}\Big\|\sum_{i,j\in J_k}X_{ij}s_it_j\Big\|_{2^k}\Big)
\\
&\lesssim a+\sup_{\|s\|_2,\|t\|_2\leq 1}\Big\|\sum_{i,j\notin I_{2k-3}}X_{ij}s_it_j\Big\|_{2^k}
\lesssim_\alpha a+\gamma,
\end{align*}
where the first equivalence follows by comparison of weak and strong moments \cite[Theorem 1.1]{LaSt} and the last one by the construction of $I_k$.
Therefore
\[
\Ex\|U\|=\Ex\sup_{k\geq 1}\|U_k\|\lesssim 
\sup_{k\geq 1}\Big(\Ex\|U_k\|^{2^k}\Big)^{1/2^k}\lesssim_\alpha a+\gamma.
\] 
In a similar way we show that $\Ex\|V\|\leq a+\gamma$. Finally observe that if $(i,j)\in E_3$
and $M_k\leq i<N_{k+1}$ then either $j\leq N_k$ or $j\geq M_{k+1}$ and if $N_{k+1}\leq i<M_{k+1}$
then $j\leq N_k$ or $j\geq M_{k+2}$. In all this cases $\|X_{ij}\|_2\leq a2^{-2^{k-2}}$ by \eqref{eq:estoffdiag}.
Thus \eqref{eq:th44} yields
\[
\Ex \|W\|\lesssim_\alpha a+\sup_k\log^\beta (M_{k+1})\max_{M_k\leq i<M_{k+1}}\max_j\|X_{ij}\|_2
\lesssim_\alpha a+\sup_k2^{k\beta}a2^{-2^{k-2}}\lesssim_\alpha a.
\]

\end{proof}


\noindent
{\sc Rafa{\l} Lata{\l}a}\\
Institute of Mathematics\\
University of Warsaw\\
Banacha 2\\
02-097 Warszawa, Poland\\
\texttt{rlatala@mimuw.edu.pl}

\medskip

\noindent
{\sc Witold Świątkowski}\\
Institute of Mathematics\\
Polish Academy of Sciences\\
Śniadeckich 8\\
00-656 Warszawa, Poland\\
\texttt{wswiatkowski@impan.pl}


\begin{thebibliography}{99}
\bibitem{AGZ}
 G.~W.~Anderson, A.~Guionnet, and O.~Zeitouni,
\textit{An introduction to random matrices}, Cambridge University Press, Cambridge, 2010.






\bibitem{BBvH}
A.~S.~Bandeira, M.~T.~Boedihardjo, and R.~van Handel,
\textit{Matrix concentration inequalities and free probability},  arXiv:2108.06312.



\bibitem{DMS} S.~J.~Dilworth and S.~J.~Montgomery-Smith, 
\textit{The distribution of vector-valued Rademacher series}, 
Ann. Probab. \textbf{21} (1993), 2046--2052. 



\bibitem{Ha} L.~H.~Harper, 
\textit{Optimal assignments of numbers to vertices,}
J. Soc. Indust. Appl. Math. \textbf{12} (1964), 131--135.

\bibitem{Hi} P.~Hitczenko,
\textit{Domination inequality for martingale transforms of a Rademacher sequence}, 
Israel J. Math. \textbf{84} (1993) ,161--178.

\bibitem{HK} P.~Hitczenko and S.~Kwapie\'n, 
\textit{On the Rademacher series}, 
in: Probability in Banach Spaces, 9, Sandjberg, Denmark, 31--36. Birkh\"auser, Boston, 1994.



\bibitem{LaSt} R.~Lata{\l}a and M.~Strzelecka,
\textit{Comparison of weak and strong moments for vectors with independent coordinates},
Mathematika \textbf{64} (2018), 211--229.

\bibitem{HLY}  R.~Lata{\l}a, R.~van~Handel and P.~Youssef, 
\textit{The dimension-free structure of nonhomogeneous random matrices}, 
Invent. Math. \textbf{214} (2018), 1031--1080.

\bibitem{LT} M.~Ledoux and M.~Talagrand,
\textit{Probability in Banach Spaces},
Springer-Verlag, Berlin, 1991.

\bibitem{MS} S.~J.~Montgomery-Smith,  
\textit{The distribution of Rademacher sums},
Proc. Amer. Math. Soc. \textbf{109} (1990), 517--522.

\bibitem{Se} Y.~Seginer,
\textit{The expected norm of random matrices}, 
Combin. Probab. Comput. \textbf{9} (2000), 149--166.

\bibitem{Tao} T. Tao, 
\textit{Topics in random matrix theory},
American Mathematical Society, Providence, RI, 2012.

\bibitem{Tr} J.~A.~Tropp,
\textit{An introduction to matrix concentration inequalities}, 
Foundations and Trends in Machine Learning \textbf{8} (2015), 1--230.

\bibitem{vHstr} R.~van Handel,
\textit{Structured random matrices}, in Convexity and concentration, 107--156, Springer, New York, 2017.


\end{thebibliography}
\end{document}